\documentclass{amsart}
\usepackage{amssymb,mathtools,hyperref,enumitem}
\usepackage{microtype}
\setlist[enumerate,1]{label=\textnormal{(\roman*)}}
\setlist[itemize,1]{label={--}}
\usepackage{inconsolata}
\usepackage{fourier}

\numberwithin{equation}{section}

\usepackage[ left = 1.3in, right = 1.3in,
             top = 1.2in, bottom = 1.2in ]{geometry}

\usepackage{tikz}
\usetikzlibrary{decorations.markings}
\tikzset{->-/.style={decoration={markings, mark=at position .5 with {\arrow{>}}}, postaction={decorate}}}
\tikzset{-<-/.style={decoration={markings, mark=at position .5 with {\arrow{<}}}, postaction={decorate}}}

\DeclareMathOperator{\val}{val}
\DeclarePairedDelimiter{\abs}{\lvert}{\rvert}

\newtheorem{theorem}[subsection]{Theorem}
\newtheorem{definition}[subsection]{Definition}
\newtheorem{lemma}[subsection]{Lemma}

\theoremstyle{remark}
\newtheorem{example}[subsection]{Example}
\newtheorem{remark}[subsection]{Remark}

\newcommand{\K}{\mathsf{k}}
\newcommand{\X}{\mathsf{x}}
\newcommand{\Y}{\mathsf{y}}
\renewcommand{\H}{\mathsf{h}}

\newcommand{\KK}{\mathsf{K}}
\newcommand{\KKb}{\overline{\mathsf{K}}}
\newcommand{\XX}{\mathsf{X}}
\newcommand{\YY}{\mathsf{Y}}

\newcommand{\myss}[1]{\noindent\underline{\emph{#1}}:\:}
\newcommand{\explain}[1]{\mbox{(#1)}}
\newcommand{\dmrjdel}[1]{}

\newcommand{\al}{\alpha}
\newcommand{\be}{\beta}

\newcommand{\De}{\Delta}
\newcommand{\varep}{\varepsilon}
\newcommand{\lam}{\lambda}

\newcommand{\id}{\mathrm{id}}

\newcommand{\rar}{\rightarrow}
\newcommand{\fsl}{\mathfrak{sl}}

\newcommand{\oc}{\circ}
\newcommand{\ot}{\otimes}

\newcommand{\half}{\frac{1}{2}}
\newcommand{\q}{q}
\newcommand{\qb}{\overline{q}}	% \bar ok
\newcommand{\Kb}{\overline{\K}}

\newcommand{\Qbinom}[2]{\left[\begin{array}{c} \!\!\! #1 \!\!\!  \\  \!\!\! #2 \!\!\! \end{array}\right]}

\newcommand{\Exp}[1]{\mathrm{Exp}_{#1}}
\newcommand{\ehHtH}{e^{\frac{h}{4}\, \H\ot \H  }}

\newcommand{\fa}{\mathfrak{a}}

\newcommand{\cA}{\mathcal{A}}

\newcommand{\fb}{\mathfrak{b}}
\newcommand{\cB}{\mathcal{B}}
\newcommand{\bbC}{\mathbb{C}}
\newcommand{\cD}{\mathcal{D}}

\newcommand{\bI}{\mathbf{I}}
\newcommand{\cK}{\mathcal{K}}

\newcommand{\bM}{\mathbf{M}}

\newcommand{\bbQ}{\mathbb{Q}}
\newcommand{\bbR}{\mathbb{R}}
\newcommand{\cS}{\mathcal{S}}
\newcommand{\sR}{\mathsf{R}}
\newcommand{\sS}{\mathsf{S}}
\newcommand{\fS}{\mathfrak{S}}

\newcommand{\cU}{\mathcal{U}}
\newcommand{\bu}{\mathbf{u}}

\newcommand{\bv}{\mathbf{v}}
\newcommand{\bbZ}{\mathbb{Z}}

\newcommand{\bm}{\mathbf{m}}
\newcommand{\bfe}{\mathbf{e}}

\newcommand{\Uhsltwo}{\cU_h(\fsl_2)}
\newcommand{\Uhsl}[1]{\cU_h(\fsl_{#1})}
\newcommand{\qr}{q^{\frac{1}{2}}}
\newcommand{\qbr}{\qb^{\frac{1}{2}}}
\makeatletter
\newcommand*{\coloneqq}{\mathrel{\rlap{%
           \raisebox{0.3ex}{$\m@th\cdot$}}%
           \raisebox{-0.3ex}{$\m@th\cdot$}}%
           =}
\makeatother
\makeatletter
\newcommand*{\eqqcolon}{=\mathrel{\rlap{%
           \raisebox{0.3ex}{$\m@th\cdot$}}%
           \raisebox{-0.3ex}{$\m@th\cdot$}}%
       }
\makeatother

\newcommand{\Wterm}[2]{\frac{{#1}^{#2}}{[#2]!}}

\title[Combinatorial aspects of $\Uhsl{n+1}$]{Combinatorial aspects of the quantized universal enveloping algebra of $\fsl_{n+1}$}
\author[R. Cheng]{Raymond Cheng}
\address{Department of Mathematics \\
  Columbia University \\
  2990 Broadway \\
  New York, NY, USA 10027
}
\email{rcheng@math.columbia.edu}

\author[D.~M.~Jackson]{David~M.~Jackson}
\address{Department of Combinatorics and Optimization \\
  University of Waterloo \\
  200 University Avenue West \\
  Waterloo, ON, Canada  N2L 3G1
}
\email{dmjackso@uwaterloo.ca}

\author[G.~J.~Stanley]{Geoff~J.~Stanley}
\address{Department of Physics \\
  Oxford University \\
  Parks Road \\
  Oxford, UK OX1 3PU
}
\email{geoff.stanley@physics.ox.ac.uk}

\date{July 2, 2018}
\begin{document}

\begin{abstract}
% Wednesday 14 June 2017:
Quasi-triangular Hopf algebras were introduced by Drinfel'd in his construction of
solutions to the Yang--Baxter Equation.
This algebra is built upon $\mathcal{U}_h(\mathfrak{sl}_2)$, the
quantized universal enveloping algebra of the Lie algebra $\mathfrak{sl}_2$.
In this paper, combinatorial structure in $\mathcal{U}_h(\mathfrak{sl}_2)$ is
elicited, and used to assist in highly intricate calculations in this algebra.
To this end, a combinatorial methodology is formulated for straightening
algebraic expressions to a canonical form in the case $n=1$.
We apply this formalism to the quasi-triangular Hopf algebras and obtain a
constructive account not only for the derivation of the Drinfel'd's $R$-matrix,
but also for the arguably mysterious ribbon elements of $\mathcal{U}_h(\mathfrak{sl}_2)$.
Finally, we extend these techniques to the higher dimensional algebras
$\mathcal{U}_h(\mathfrak{sl}_{n+1})$.
While these explicit algebraic results are well-known, our contribution is in
our formalism and perspective: our emphasis is on the combinatorial structure
of these algebras and how that structure may guide algebraic constructions.
\end{abstract}
\maketitle

\section{Introduction}

\subsection{Motivation: Knot Theory}
%-------------------------------------------------------------------------------------------------
A rich setting in which quasi-triangular Hopf algebras appear is knot theory, so
we shall begin by explaining very briefly and informally some of the
background to this. A \emph{knot} is an embedding of the unit circle into \(\bbR^3\)
and two knots are equivalent if one may be transformed into the other smoothly:
that is, without cutting and re-attaching the ends. An essential question in knot
theory is how to construct a map \(\theta\colon \cK\rar \cS\), from the set \(\cK\)
of all knots to a set \(\cS\) such that if \(\fa\) and \(\fb\) are knots, then
\(\theta(\fa) \neq \theta(\fb)\) implies that \(\fa\) and \(\fb\) are \emph{inequivalent}
knots.  The map \(\theta\) is called a \emph{knot invariant}.

The discovery in the 1990's that the Yang--Baxter Equation, which appeared in
mathematical physics, also arose in knot theory prompted a remarkable
resurgence of activity in knot theory and, \emph{obiter dictu}, marked the
beginning of what is now commonly termed \emph{Modern Knot Theory}.
%
% NEW
The appearance of the Yang--Baxter Equation may be seen as follows.
An oriented knot may be represented in the plane by its regular projection as a
four regular graph, together with marks attached to each vertex to indicate
whether a crossing is positive or negative.
Such an object is called a knot diagram.
A result of Alexander~\cite{al} shows that a knot diagram may be viewed as the
closure of a braid by selecting a base point in the plane of a knot diagram,
and then using the (isotopy preserving) Reidemeister Moves in such a way that
each segment of the diagram between successive vertices is directed in the
anti-clockwise sense around the base point.
The braid, in turn, may be expressed as a product of the braid generators
\(s_1, \ldots, s_{n-1}\), for an \(n\)-stranded braid, where \(s_i\) is the
transposition \((i,i+1)\), together with the braid relations, of which one is
\(s_{i} s_{i+1} s_{i} =  s_{i+1} s_{i} s_{i+1}\) for \(i=1,\ldots,n-1\).
This accounts for the appearance of (matrix) representations of the braid
group.
The Yang--Baxter Equation
\begin{equation}\label{e:YBequ}
  (\bM\ot\bI)(\bI\ot\bM)(\bM\ot\bI) = (\bI\ot\bM)(\bM\ot\bI)(\bI\ot\bM)
\end{equation}
is the image of this relation in the matrix representation.
Such a matrix \(\bM\) is called an \emph{\(sR\)-matrix}.
Solutions of the Yang--Baxter equation may be obtained through \emph{Ribbon Hopf Algebras}.
In general, such algebras are difficult to construct.
The remarkable work of Drinfel'd and Jimbo in the late 1980's showed that every
semisimple Lie algebra over \(\bbC\) gives rise to such an algebra, the
starting point of which is the \emph{quantized universal enveloping algebra} of
a semisimple Lie algebra.
Consequently, many new knot invariants, generally contained in the class of
\emph{quantum invariants}, were discovered.
Readers interested in reading further about the connexions with knot theory are
referred to~\cite{oh}.

The three algebras which will be encountered are:

\begin{enumerate}
\item  the \emph{quantized universal enveloping algebra} \(\Uhsl{2}\) of the
Lie algebra \(\fsl_2\); this is a Hopf algebra:

\item a \emph{quasi-triangular Hopf algebra};
this is a Hopf algebra with an invertible element \(\sR\) called a universal
\(\sR\)-matrix;

\item a \emph{ribbon Hopf algebra};
this is a quasi-triangular Hopf algebra with a particular element $\bv$ called a \emph{ribbon element} (determined from \(\sR\)).
\end{enumerate}

\subsection{Purpose}
We include a self-contained introduction to quantized universal enveloping
algebras of semisimple Lie groups from a particular perspective: namely, that
they contain a rich combinatorial structure which may be used as a guide to
highly intricate algebraic calculations within these algebras.
We demonstrate the efficacy of this approach by deriving several fundamental
results that may be found in~\cite{cp,k,oh} and original sources such
as~\cite{d3,rt1,rt2,bu,kr90,ls91}.
These results include:
\begin{itemize}
\item [--] straightening in \(\Uhsl{2}\), but from a constructive approach;
  \hfill (\S\ref{S:straighten.commutator})

\item[--]
  a direct derivation of an \(\sR\)-matrix for \(\Uhsl{2}\) without recourse to Drinfel'd's \\
  quantum double or to the quantum Weyl group;
  \hfill (\S\ref{SS:ConstRmorph})

\item[--] the same for \(\Uhsl{n+1}\), \(n \geq 2\);
 \hfill (\S\S\ref{RmatrixSLN1}-\ref{RmatrixSLN2})

\item[--]
  a direct, essentially combinatorial, construction of the ribbon Hopf structure \\
  on \(\Uhsl{2}\);
  \hfill (\S\ref{S:ribbon}).
\end{itemize}

After completing  this investigation, it came to our attention that the
article~\cite{kt91} studied universal enveloping algebras and a universal
\(\sR\)-matrix for quantized super algebras.
They did so through the combinatorics of root systems.
While there are some similarities, our approach is through the combinatorics of
straightening and the combinatorics of \(q\)-series.

%%%%%%%%%%%%%%%%%%%%%%%%%%%%%%

\subsection{Organization}
This article is organized as follows:

In \S\ref{S:QUEA}, we discuss straightening in \(\Uhsl{2}\) and establish the
technical lemmas which are crucial to all that follows.
These are applied to straighten the monomial \(\X^a\Y^b\) in \(\Uhsl{2}\) so
that it is a sum of monomials of the form \(\Y^c\X^d\).
We also discuss some \(q\)-identities in the combinatorial context of
inversions in bimodal permutations, and then prove an extension of a classic
identity of Cauchy that is crucial to our approach to the construction of the
ribbon Hopf structure on \(\Uhsl{2}\).

In \S\ref{S:Rmorphism}, we constructively derive an \(\sR\)-matrix for
\(\Uhsl{2}\).

In \S\ref{S:ribbon}, we give an explicit construction for ribbon element in
\(\Uhsl{2}\).

In \S\ref{S:Uhsln}, we extend the straightening framework developed for
\(\Uhsl{2}\) to \(\Uhsl{n+1}\) for \(n \geq 2\).
These higher dimensional studies further clarify, and amplify, the essential
features of our technique.

In \S\ref{S:Uhsl3-coeffs} we derive a \(\sR\)-matrix for \(\Uhsl{n+1}\).

\subsubsection*{Acknowledgements}
DMJ would like to thank Pavel Etingof of useful discussions.
We wish to thank an anonymous referee for most valuable suggestions, and an
assiduous reading of the paper.
RC and DMJ were supported by The Natural Sciences and Engineering Research
Council of Canada.

\section{Straightening in the Quantized Universal Enveloping Algebra of \texorpdfstring{$\fsl_2$}{sl2}}\label{S:QUEA}
In this Section, we develop a framework for performing \emph{straightening}
computations in the Drinfel'd--Jimbo quantized universal enveloping algebra
\(\Uhsltwo\) for the Lie algebra \(\fsl_2\).
To begin, recall that \(\Uhsltwo\) is the associative \(\bbC[\![h]\!]\)-algebra
with underlying \(\bbC[\![h]\!]\)-module \(\cU(\fsl_2)[\![h]\!]\) formal power
series in \(h\) with coefficients in the universal enveloping algebra
\(\cU(\fsl_2)\) of the Lie algebra \(\fsl_2\), and product determined by the
following degree 2 relations in the generators \(\X,\Y,\H \in \Uhsl{2}\):
\[
  \H\X - \X\H = 2\X, \quad
  \H\Y - \Y\H = -2\Y, \quad
  \X\Y - \Y\X = \frac{e^{\frac{h}{2}\H} - e^{-\frac{h}{2}\H}}{e^{\frac{h}{2}} - e^{-\frac{h}{2}}}.
\]
Our point of departure is to rearrange these relations into the form
\begin{equation}\label{D:QhUSL.rels}
  \X\H = \H\X - 2\X, \quad
  \Y\H = \H\Y + 2\Y, \quad
  \X\Y = \Y\X + \frac{e^{\frac{h}{2}\H} - e^{-\frac{h}{2}\H}}{e^{\frac{h}{2}} - e^{-\frac{h}{2}}},
\end{equation}
so that they may be interpreted as a means for transforming arbitrary monomials
in \(\X\), \(\Y\), and \(\H\) into sums of monomials in which the generators
are ordered \(\H \prec \Y \prec \X\).
This is made precise by the Poincar\'e--Birkhoff--Witt Theorem, which says that
the set
\[ \cB_h \coloneqq \{\H^r\Y^s\X^t : a,b,c \in \bbZ_{\geq 0}\} \]
is a \(\bbC[\![h]\!]\)-basis for \(\Uhsl{2}\).
See~\cite[p.199]{cp} for a proof.

The Poincar\'e--Birkhoff--Witt Theorem allows us to make explicit comparisons
between elements of \(\Uhsl{2}\) by expressing elements in terms of the basis
\(\cB_h\) and then by comparing the coefficients of corresponding basis elements.
The process of writing an element of \(\Uhsl{2}\) in terms of the basis
\(\cB_h\) is referred to as \emph{straightening}.
Much of our work in this article is to codify straightening in quantized
universal enveloping algebras, to show how a robust formalism for straightening
can be used to perform detailed calculations, and also to show how to build
such a framework from a constructive point of view.

\subsection{Notation}\label{SS:quantum-notation}
We use the following notation throughout the article:
\begin{equation}\label{e:NTN:qKqb}
  q      \coloneqq e^{\frac{h}{2}}, \quad
  \K     \coloneqq e^{\frac{h}{4}\H}, \quad
  \qb    \coloneqq q^{-1}, \quad
  \Kb    \coloneqq \K ^{-1}, \quad
  [\H+n] \coloneqq \frac{q^n \K ^2 - \qb^n \Kb^2}{q -\qb}.
\end{equation}
With this notation, the commutation relation for \(\X\) and \(\Y\)
from~\eqref{D:QhUSL.rels} is simply expressed as
\begin{equation}\label{e:XYH}
  \big[\X,\Y\big] \coloneqq \X\Y - \Y\X = \frac{\K^2 - \Kb^2}{q - \qb} = [\H].
\end{equation}
For integers \(n\) and \(k\), write
\[
  [n]_q           \coloneqq \frac{q^n - \qb^n}{q-\qb}, \quad
  [n]!_q          \coloneqq  \prod_{i = 1}^n [i]_q, \quad
  \Qbinom{n}{k}_q \coloneqq \frac{[n]!_q}{[k]!_q \, [n-k]!_q}, \quad
  [n]_{q,i} \coloneqq \frac{[n]!_q}{[n-i]!_q},
\]
for the \emph{quantum integer} \(n\), the \emph{quantum factorial function},
the \emph{quantum binomial function}, and the \emph{quantum lower factorial}
respectively, where \(k\) is a non-negative integer.
Note that subscript \(q\) in each notation is to be thought of as the argument
in its definition.
That is, for any function \(f(q)\) of \(q\), we write
\begin{align*}
  [n]_{f(q)} \coloneqq \frac{f(q)^n - f(\qb)^n}{f(q)-f(\qb)},
\end{align*}
and similarly with the other definitions.
When an explicit subscript is omitted, \(f(q) = q\) by convention.

Finally, for non-negative integers \(i\),
\begin{equation}\label{e:BinomH}
  [\H + n]_{(i)} \coloneqq \prod_{r=0}^{i-1} [\H +n - r],
  \qquad\mbox{and}\qquad
  \Qbinom{\H + n}{i} \coloneqq \frac{1}{[i]!} [\H + n]_{(i)}.
\end{equation}

\subsection{Basic Straightening Rules}\label{SS:Uhsl2Straightening}
The proofs of the following are straightforward and are largely omitted.

\begin{lemma}[Separation Lemma]\label{L:xya:qId}
Let \(x\), \(y\), and \(a\) be indeterminates.
Then
\begin{enumerate}
  \item \label{L:xya:qId.i}   \([-x] = -[x]\),
  \item \label{L:xya:qId.ii}  \([x] \, [y]  - [x-a] \, [y+a] = [a]\,[y-x+a]\),
  \item \label{L:xya:qId.iii} \([x] \, [y]  +  [a]\,[x+y+a] = [x+a]\,[y+a]\).
\end{enumerate}
\end{lemma}
\begin{proof}
Identity~\ref{L:xya:qId.ii} comes by noting
\((q - \qb)^2\, [x] \, [y] =
  \big(q^{x+y} +\qb^{x+y}\big) - \big(q^{y-x} +\qb^{y-x}\big)\).
Applying this twice,
\[
  (q - \qb)^2\, \big( [x] \, [y] -  [x-a] \, [y+a] \big)
    = q^{y-x} \big(q^{2a} - 1\big) + \qb^{y-x} \big(\qb^{2a} - 1\big).
\]
Finally,~\ref{L:xya:qId.iii} follows from~\ref{L:xya:qId.ii} upon replacing
\(x\) by \(-x\).
\end{proof}

\begin{remark}
These identities are valid when either:
\begin{enumerate}
  \item[(a)] both \(x\) and \(y\) are integers; or
  \item[(b)] one of \(x\) or \(y\) is an integer and the other is an
    expression involving \(\H\).
\end{enumerate}
Lemma~\ref{L:xya:qId}\ref{L:xya:qId.ii} may be viewed as a device for
separating the \(x\) and \(y\) in \([y-x-a]\).
\end{remark}

\begin{lemma}[Straightening]\label{L:calcUhSL2}
Let \(a \in \bbZ\), \(b, c \in \bbZ_{\geq 0}\), and \(f(x)\) be a formal power
series in \(x\).
Then the following hold in \(\Uhsltwo\).
\begin{align}
  \label{L:calcUhSL2.ii}\tag{i}
    \X^b\, f(\H) & = f(\H-2b)\X^b,
      &
    \Y^b\, f(\H) & = f(\H+2b)\Y^b, \\
  \label{L:calcUhSL2.iii}\tag{ii}
    \X^b\K^a & = \qb^{ab}\K^a\X^b,
      &
    \Y^b \K^a & = q^{ab} \K^a \Y^b, \\
  \label{L:calcUhSL2.iv}\tag{iii}
    (\K\X)^b & = \qb^{\frac{1}{2}b(b-1)} \K^b \X^b,
      &
    (\K\Y)^b  & = q^{\frac{1}{2}b(b-1)} \K^b\Y^b, \\
  \label{L:calcUhSL2.v}\tag{iv}
    \X\,\Y^b  & = \Y^b\X + [b]\,[\H+b-1]\,\Y^{b-1},
      &
    \X^b\Y & = \Y\,\X^b + [b]\,[\H-b+1]\,\X^{b-1}, \\
  \label{L:calcUhSL2.vi}\tag{v}
    (\K\X)^b f(\H) & = f(\H-2b)\:(\K\X)^b,
      &
    (\Kb\Y)^b f(\H) & = f(\H+2b) \: (\Kb\Y)^b, \\
  \label{L:calcUhSL2.vii}\tag{vi}
    \X^b(\K\X)^c & = \qb^{\frac{1}{2}(c^2+2bc-c)} \K^c \X^{b+c},
      &
    \Y^b(\K\Y)^c & = q^{\frac{1}{2}(c^2+2bc-c)} \K^c \Y^{b+c}.
\end{align}
\end{lemma}
\begin{proof}
  We only prove~\eqref{L:calcUhSL2.v}.
  Let \(A_b \coloneqq \X\Y^b\).
  Then, from~\eqref{e:XYH} and part~\eqref{L:calcUhSL2.ii},
  \begin{equation}\label{Aby}
    A_b = (\X\Y)\Y^{b-1} = \Y A_{b-1} + [\H]\Y^{b-1}.
  \end{equation}
  Iterating this gives
  \begin{equation}\label{Abfbh}
    A_b = \Y^b\X + f_b(\H) \Y^{b-1} \quad\mbox{where}\quad f_0(\H)=0
  \end{equation}
  which, when substituted into~(\ref{Aby}) and part~\eqref{L:calcUhSL2.ii} is applied,
  gives \(f_b(\H) = f_{b-1}(\H+2) + [\H]\). Thus
  \[f_b(\H) = \sum_{i = 0}^{b - 1} [\H + 2i] = f_{b-1}(\H) + [\H+2b-2].\]
  Lemma~\ref{L:xya:qId}\ref{L:xya:qId.iii} with \(a = 1\), \(x = b - 1\), and
  \(y =\H + b - 2\) gives \([\H+2b-2] = [b]\,[\H+b-1] - [b-1]\,[\H+b-2]\), so
  \[f_b(\H) - [b]\,[\H+b-1] = f_{b-1}(\H) - [b-1]\,[\H+b-2] = c \]
  where \(c\) is therefore independent of \(b\).
  Setting \(b=0\) in the left hand side gives \(c = f_0(\H) = 0\),
  so \(f_b(\H) = [b]\,[\H+b-1]\) and the result follows from~\eqref{Abfbh}.
\end{proof}

When constructing an \(\sR\)-matrix in \S\ref{SS:ConstRmorph}, a term
\(e^{\frac{h}{4}\H \otimes \H}\) will appear.
The following result will be useful for commuting terms past this exponential.
\begin{lemma}\label{L:1otX}
Let \(f(x)\) be a formal power series in \(x\) and \(m \in \mathbb{Z}\). Then
\begin{enumerate}
  \item \label{L:1otX.i} \(f(1\otimes \K^m\X) \ehHtH =  \ehHtH f(\Kb^2\ot \K^m\X)\); and
  \item \label{L:1otX.ii}\(f(1\otimes \K^m\Y) \ehHtH =  \ehHtH f(\K^2\ot \K^m\Y)\).
\end{enumerate}
\end{lemma}

\subsection{Straightening of \texorpdfstring{$\X^a\Y^b$}{XaYb}}\label{S:straighten.commutator}
We now straighten \(\X^a\Y^b\), \(a,b \in \bbZ_{\geq 0}\), with respect to the
ordering \(\H \prec \Y \prec \X\) by a constructive method.
From Lemma~\ref{L:calcUhSL2}\eqref{L:calcUhSL2.v}, the straightening of the
premultiplication of \(\Y^b\) by \(\X\) is
\[\X \,\Y ^b = \Y ^b \X  + [b]\, [\H +b-1]\, \Y ^{b-1}.\]
Iterating this \(a\) times, and noting that \(\X\) may be moved through quantum
brackets containing only \(\H\) by means of
Lemma~\ref{L:calcUhSL2}\eqref{L:calcUhSL2.ii},
\begin{equation}\label{F1}
  \X^a\Y^b = \sum_{0 \leq i,j \leq \min(a,b)} F_{a,b,i,j}(\H)\, \Y^i \X^j.
\end{equation}
But \(e^\H\,\X^a\Y^b = \X^a\Y^b\, e^{\H+2a-2b}\) by
Lemma~\ref{L:calcUhSL2}\eqref{L:calcUhSL2.ii} and so, applying this
to~(\ref{F1}), gives
\[
  \X^a\Y^b e^{\H+2a-2b} =
    \sum_{0 \leq i,j \leq \min(a,b)} F_{a,b,i,j}(\H)\, \Y^i \X^j e^{\H+2j-2i},
\]
and so
\[
  \X^a\Y^b =
    \sum_{0 \leq i,j \leq \min(a,b)} F_{a,b,i,j}(\H)\, \Y^i \X^j e^{2(j-a)-2(i-b)}.
\]
Equating coefficients of \(\Y^i \X^j\) on the right hand side of this
and~\eqref{F1}, we have \(i=b-k\) and \(j=a-k\) for some non-negative integer
\(k\), whence we conclude that
\begin{equation}\label{F2}
  \X^a\Y^b = \sum_{0\le k \le \min(a,b)} G_{a,b,k}(\H )\, \Y^{b-k}\X^{a-k}
\end{equation}
where \(G_{a,b,k}(\H) \coloneqq F_{a,b,a-k,b-k}(\H)\). Since commuting \(\X\)
from the left of \(\Y^b\) yields a single term of top degree with coefficient
\(1\), the boundary condition is
\begin{equation}\label{F3}
  G_{a,b,0}(\H) = 1.
\end{equation}

A recursion for \(G_{a,b,k}\) is obtained from the identity
\(\X^a\Y^b= \X ^{a-1} \left(\X \, \Y ^b\right)\).
First, from Lemma~\ref{L:calcUhSL2}\eqref{L:calcUhSL2.ii}
and~\eqref{L:calcUhSL2.v},
\[
  \X^a\Y^b =
    \big(\X^{a-1}\Y^{b-1}\big)(\Y\X) + [b]\,[\H+b-2a+1]\,\big(\X^{a-1}\Y^{b-1}\big).
\]
Then, substituting~\eqref{F2} into this,
\[
  \sum_{k \geq 0} G_{a,b,k} \Y^{b-k} \X^{a-k} =
    \sum_{k \geq 0} G_{a-1,b-1,k}\big(\Y^{b-k-1} \X^{a-k-1}(\Y\X) + [b]\, [\H +b-2a+1] \Y^{b-k-1}\X^{a-k-1}\big).
\]
From Lemma~\ref{L:calcUhSL2}\eqref{L:calcUhSL2.v},
\(\X^{a-k-1}\Y = \Y\,\X^{a-k-1} + [a-k-1] \, [\H-a+k+2] \, \X^{a-k-2}\).
Substituting this into the above, and then equating the coefficients of
\(\Y^{b-k}\X^{a-k}\) gives the recurrence equation
\[ G_{a,b,k} = G_{a-1,b-1,k} + \big([a-k] \,  [\H +2b-a-k+1]+ [b] \,  [\H +b-2a+1]\big) G_{a-1,b-1,k-1}. \]
Then from Lemma~\ref{L:xya:qId}\ref{L:xya:qId.iii}, with \(x\mapsto b\),
\(y \mapsto \H +b-2a+1\) and \(a \mapsto a-k\),
\begin{equation}\label{F4}
  G_{a,b,k} = G_{a-1,b-1,k} + [a+b-k] \, [\H +b-a-k+1]\,G_{a-1,b-1,k-1}.
\end{equation}

Each instance of \(G_{i,j,\cdot}\) in this recurrence equation satisfies \(i-j =a-b\).
The only term that does not contain \(a-b\) is \([a+b-k]\).
This suggests using \([k]\, [a+b-k] =  [a]\, [b] - [a-k]\, [b-k]\) from the
Separation Lemma~\ref{L:xya:qId} to separate \(a\) and \(b\) in this quantum
bracket and then transforming \(G_{a,b,k}\) to form a new recurrence equation
in which \(a-b\) is an invariant.
Let
\begin{equation}\label{e:D}
  G _{a,b,k} \eqqcolon \frac{[a]_k \, [b]_k}{[k]!} \, B_{a,b,k}.
\end{equation}
Then, substituting~\eqref{e:D} into~\eqref{F4} gives
\begin{equation}\label{F6}
  [a]\,[b]\,B_{a,b,k} - [a-k]\,[b-k]\,B_{a-1,b-1,k} \\
    = \big([a]\,[b] - [a-k]\,[b-k]\big) \, [\H +b-a-k+1]\,B_{a-1,b-1,k-1}.
\end{equation}
Suppose that \(B_{a,b,k}\) depends only on the difference \(b - a\).
Then \(B_{a,b,k} = B_{a-1,b-1,k}\) and~(\ref{F6}) becomes
\[ B_{a,b,k} = [\H +b-a-k+1]\, B_{a,b,k-1} \]
for \(k \ge 1\) and \(B_{a,b,0} = 1\) from~\eqref{F3}. This suggests the
solution \(B_{a,b,k} = [\H +b-a]_k\). Indeed,
it is readily checked that this does indeed satisfy~\eqref{F6}, from which
we have
\[ G_{a,b,k} = [a]_k \, [b]_k \Qbinom{\H +b-a}{k}. \]
So, from~(\ref{F2}),  we have therefore (both derived and) proved the following:

\begin{lemma}\label{L:XaYbHq}
  Let \(a\) and \(b\) be non-negative integers. Then
  \[
    \Wterm{\X}{a} \, \Wterm{\Y}{b}
      = \sum_{i\ge0} \Qbinom{\H +b-a}{i} \Wterm{\Y}{b-i} \, \Wterm{\X}{a-i}.
  \]
  In particular, when \(a = b\),
  \[ \X^n \Y^n = \sum_{i=0}^n \, [n]_i^2  \Qbinom{\H }{i}  \,  \Y^{n-i}\X^{n-i}. \]
\end{lemma}

\subsection{Straightening $q$-commuting variables}\label{S:inversions}
Indeterminates \(a\) and \(b\) are said to \emph{\(q\)-commute} if \(ba = q\, ab\).
Certain combinations of elements in \(\Uhsl{2}\) \(q\)-commute and so we will
find use for straightening rules involving series in \(q\)-commuting variables.
Here, we collect straightening rules involving abstract \(q\)-commuting variables.
We view these identities as arising from combinatorial properties of inversions
in permutations.

For \(n,k \in \bbZ\), define the \emph{\(q\)-integer \(n\)}, the \emph{\(q\)-factorial}
of \(n\) and the \emph{\(q\)-binomial coefficient} of \(n\) and \(k\) as
\begin{equation}\label{eq:q-functions}
  (n)_q          \coloneqq \frac{1-q^n}{1-q}, \quad
  (n)!_q         \coloneqq (1)_q \cdot (2)_q \cdots  (n)_q, \quad
  \binom{n}{k}_q \coloneqq\frac{(n)!_q}{(k)!_q\, (n-k)!_q},
\end{equation}
respectively.
The \(q\)-\emph{exponential series} is
\[ \exp_q(x) \coloneqq \sum_{n\ge0} \frac{x^n}{(n)!_q} \in \bbQ(q)\, [\![x]\!] \]
as a formal power series in \(x\) with coefficients that are rational functions
of \(q\).

The \(q\)-exponential series has a multiplicative property for \(q\)-commuting
indeterminates.
The proof uses the observation that \(q\) is associated with a combinatorial
property of sets, as follows.
An \emph{ordered bipartition} \(\{1,\ldots,n\}\) of \emph{type} \((r,n-r)\) is
\((\alpha,\beta)\), where \(\alpha\) and \(\beta\) are disjoint subsets of
\(\{1,\ldots,n\}\) of size \(r\) and \(n-r\), respectively.
A \emph{between-set inversion} of \((\alpha,\beta)\) is a pair
\((i,j) \in \alpha\times \beta\) such that  \(i>j\).
An \emph{inversion} in a permutation \(\pi\in\fS_n\) is a pair \((i,j)\) with
\(1\le i < j \le n\) such that \(\pi(i)>j\).

\begin{lemma}\label{L:ExpqProd}
  Let \(a,b\) be such that \(ba = q\, ab\).
  Then \(\exp_{q}(a+b) = \exp_{q}(a)\exp_{q}(b)\).
\end{lemma}
\begin{proof}
  The \(q\)-commutation relation gives polynomials \(f_{r,n-r}(q)\)
  such that
  \begin{equation}\label{e:abqn}
    (a+b)^n = \sum_{r=0}^n f_{r,n-r}(q) a^r b^{n-r}.
  \end{equation}
  By construction, the coefficient of \(q^k\) in \(f_{r,n-r}(q)\) is the number
  of ordered bi-partitions of \(\{1,\ldots,n\}\) of type \((r,n-r)\) with
  precisely \(k\) between-set inversions.

  We now relate \(f_{r,n-r}(q)\) with generating series \(g_m(q)\) whose
  \(q^k\) coefficient is the number of elements in \(\fS_m\) with exactly \(k\)
  inversions.
  First, by considering the contribution to inversions by the symbol \(m\),
  we have the recursion
  \[ g_n(q) = g_{n-1}(q) \,(1+q+ \cdots + q^{n-1}) \qquad\mbox{for \(n \geq 1\)}\qquad \]
  with \(g_0(q)=1\), so \(g_n(q) = (n)!_q\).
  On the other hand, by considering a fixed bi-partition \((\alpha,\beta)\) of
  type \((r,n-r)\), we have
  \((n)!_q = g_r(q) \: g_{n-r}(q) \: f_{r,n-r}(q)\)
  since each inversion of \(\pi\) occurs within \(\alpha\), or within
  \(\beta\), or between \(\alpha\) and \(\beta\).
  Then \(f_{r,n-r}(q) = \binom{n}{r}_q\) and the result then follows
  immediately from~\eqref{e:abqn}.
\end{proof}

In constructing the ribbon structure on \(\Uhsl{2}\) in \S\ref{S:ribbon}, we
will need an extension of the following finite product identity due to
Cauchy~\cite{Cau}:

\begin{lemma}\label{L:Cauchy:FinProdId}
  Let \(z\) and \(q\) be indeterminates.
  Then
  \[ z^n = \sum_{k=0}^n \binom{n}{k}_q \prod_{i=0}^{k-1}\, \left(z-q^i\right). \]
\end{lemma}
\begin{proof}
  This can be deduced from the \(q\)-analogue of the Binomial Theorem; see,
  for example,~\cite[Corollary 2.6.11]{blue}.
\end{proof}

We now extend this identity to an identity of certain formal power series
by taking \(z\) to be an exponential related to \(q\).

\begin{lemma}\label{L:CauchyExp}
  Let \(h\), \(t\), and \(x\) be indeterminates such that \(q = e^\frac{h}{2}\).
  Then,
  \[
    e^{xht} = \sum^\infty_{k = 0} \binom{x}{k}_{q^2} \prod^{k - 1}_{i = 0} (e^{ht} - q^{2i})
    \in \bbQ[x,t][\![h]\!].
  \]
\end{lemma}
\begin{proof}
  The coefficient of \(h^m\) on the left is a polynomial in \(t\) and \(x\).
  On the right hand side, note that
  \[
    \val_h\Big(\prod^{k - 1}_{i = 0}(e^{ht} - q^{2i})\Big)
      = \val_h\Big(\prod^{k - 1}_{i = 0}(t - i)h\Big)
      = k,
  \]
  where \(\val_h\) extracts the exponent of the smallest power of \(h\) with
  with nonzero coefficient. So only the finitely many indices \(0 \leq k \leq
    m\) contribute to the coefficient of \(h^m\) and each index contributes a
  binomial coefficient \(\binom{x}{k}_{q^2}\). But, as a power series in \(h\),
  \(\binom{x}{k}_{q^2}\) has coefficients which are polynomial in \(x\).
  Therefore the coefficient of \(h^m\) on the right hand side of the statement
  is polynomial in \(t\) and \(x\).
  In particular, the coefficient of \(h^mt^n\) is a polynomial in \(x\) on both
  sides.
  By Lemma~\ref{L:Cauchy:FinProdId}, these polynomials in \(x\) agree for each
  positive integer and thus they must be equal as polynomials.
\end{proof}

\subsection{The quantum exponential function}
The functions defined in \S\ref{SS:quantum-notation} and \S\ref{S:inversions}
are related through
\begin{equation}\label{e:Qbin:CombBin}
  [n]_q  = \qb^{n-1}(n)_{q^2}, \quad
  [n]!_q = \qb^{\frac{1}{2}n(n-1)} (n)!_{q^2}, \quad
  \Qbinom{n}{k}_q = \qb^{k(n-k)} \binom{n}{k}_{q^2}.
\end{equation}
Note also that the quantum numbers are invariant under the substitution \(q \mapsto \qb\).

The \emph{quantum exponential function} defined by
\begin{equation}\label{e:qExpFn}
  \Exp{q}(x) \coloneqq \sum_{n\ge0} \frac{q^{\frac{1}{2}n(n-1)}}{[n]!_q}x^n
\end{equation}
enjoys an analogous multiplicative property as the \(q\)-exponential series
under quantum commutation.

\begin{lemma}\label{L:qExpFact}
  Let \(a\) and \(b\) be such that \(ba = \qb^2\, ab\).
  Then
  \[ \Exp{q}(a+b) = \Exp{q}(a) \, \Exp{q}(b). \]
\end{lemma}
\begin{proof}
  Equation~\eqref{e:Qbin:CombBin} shows
  \([n]! = q^{\frac{1}{2}n(n-1)} (n)!_{\qb^2}\) and so
  \(\Exp{q}(x) = \exp_{\qb^2}(x)\).
  Thus Lemma~\ref{L:ExpqProd} gives
  \[\Exp{q}(a + b) = \exp_{\qb^2}(a + b) = \exp_{\qb^2}(a) \, \exp_{\qb^2}(b) = \Exp{q}(a)\Exp{q}(b), \]
  as desired.
\end{proof}

\section{An \(\sR\)-matrix for \(\Uhsl{2}\)}\label{S:Rmorphism}
In this Section, we use the straightening framework developed in \S\ref{S:QUEA}
to give a direct and constructive approach to the construction of the
\(\sR\)-matrix for \(\Uhsl{2}\).
We begin with a few recollections;
further details can be found in~\cite{k,cp,oh}.
Recall that a \emph{Hopf algebra} over \(\bbC\) is an associative
\(\bbC\)-algebra \(\cA\) together with algebra homomorphisms
\(\Delta \colon \cA \to \cA \otimes \cA\), a \emph{co-product}, and
\(\varep \colon \cA \to \bbC\), a \emph{co-unit},
and an algebra anti-homomorphism \(\sS \colon \cA \to \cA\), an \emph{antipode},
satisfying certain relations.
In particular, as discovered by Sklyanin~\cite{sk}, \(\Uhsltwo\) is a Hopf
algebra with structure maps defined on algebra generators by
\begin{align*}
  \Delta: \Uhsltwo & \to \Uhsltwo\otimes\Uhsltwo,      & \sS: \Uhsltwo & \to \Uhsltwo, & \varep: \Uhsltwo & \to \bbC, \\
          \X & \mapsto \X \otimes \K + \Kb \otimes \X, &      \X & \mapsto -q\X,       &               \X & \mapsto 0, \\
          \Y & \mapsto \Y \otimes \K + \Kb \otimes \Y, &      \Y & \mapsto -\qb\Y,     &               \Y & \mapsto 0, \\
          \H & \mapsto \H \otimes 1  +  1  \otimes \H, &      \H & \mapsto -\H,        &               \H & \mapsto 0.
\end{align*}
Our goal is to enrich this Hopf structure of \(\Uhsltwo\) to a
quasi-triangular structure:

\begin{definition}\label{D:QTHA}
A \emph{quasi-triangular Hopf algebra} is a Hopf algebra \(\cA\) equipped
with an invertible element \(\sR \in \cA \otimes \cA\), called a
\emph{universal \(\sR\)-matrix}, satisfying
\begin{enumerate}
  \item\label{D:QTHA.i}
    \((\tau\oc\Delta)(a) = \sR \cdot \De(a) \cdot \sR^{-1}\) for every \(a \in \cA\),
  \item\label{D:QTHA.ii}
    \((\Delta\ot\id)(\sR) = \sR_{13}\cdot\sR_{23}\), and
  \item\label{D:QTHA.iii}
    \((\id\ot\Delta)(\sR) = \sR_{13}\cdot \sR_{12}\).
\end{enumerate}
\end{definition}

\noindent
Here, \(\tau \colon a \otimes b \mapsto b \otimes a\) is the twist map, and
\[
  \sR_{12} \coloneqq \sum\nolimits_i \al_i \ot \be_i \ot 1, \quad
  \sR_{13} \coloneqq \sum\nolimits_i \al_i\ot 1 \ot \be_i, \quad
  \sR_{23} \coloneqq \sum\nolimits_i 1 \ot\al_i \ot \be_i
\]
where \(\alpha_i,\beta_i \in \cA\) and are defined through writing \(\sR\) as
\(\sR \coloneqq \sum\nolimits_i \alpha_i \otimes \beta_i\).

\subsection{Action on \(\K\)}\label{S:DSeK}
We compute the Hopf algebra maps \(\Delta\), \(\varep\), and \(\sS\) on \(\K\).
Since \(\Delta\) is an algebra map
\begin{align*}
\De(\K) & = e^{\frac{h}{4}\De(\H)} = e^{\frac{h}{4}(\H \ot 1 + 1\ot \H)}
%         = \left(e^{\frac{h}{4}(\H \ot 1)}\right) \left(e^{\frac{h}{4}(1 \ot \H)}\right)
          = \left(e^{\frac{h}{4}\H} \ot 1\right)\left(1 \ot e^{\frac{h}{4}\H}\right)
         = e^{\frac{h}{4}\H} \ot e^{\frac{h}{4}\H} = \K \ot \K
\end{align*}
since \(\H \ot1\) and \(1 \ot \H\) commute.
Similarly,
\begin{align*}
  \sS(\K) = \sS\left(e^{\frac{h}{4}\H}\right) = e^{\frac{h}{4}\sS(\H)} = e^{-\frac{h}{4}\H} = \Kb,
    \quad\text{and}\quad
  \varep(\K) = \varep\left(e^{\frac{h}{4}\H}\right)= e^{\frac{h}{4}\varep(\H)} = e^0 = 1.
\end{align*}
In summary: \(\Delta(\K) = \K \otimes \K\), \(\sS(\K) = \Kb\), and \(\varep(\K) = 1\).

\subsection{Constructing an \texorpdfstring{$\sR$}{R}-matrix}\label{SS:ConstRmorph}
We construct a universal \(\sR\)-matrix for \(\Uhsltwo\) in two steps.
First, by considering the dependency on \(\X\), \(\Y\) and \(\H\) in an
\(\sR\)-matrix, we propose an ansatz.
Second, coefficients and parameters in the ansatz are determined through the
requirements on \(\sR\).

\subsubsection{An ansatz for \(\sR\)}\label{SSS:R-ansatz}
In general, \(\sR\) is a sum of elements of \(\Uhsltwo \otimes \Uhsltwo\), so
after straightening each tensor component, \(\sR\) can be expressed as a sum
of terms of the form \(\X^m\Y^t \ot \X^u\Y^n\), \(m,t,u,n \in \bbZ_{\geq 0}\).
Factors of \(\H\) appear through powers of \(\K\) and \(\Kb\) when
straightening, say in applying Lemma~\ref{L:XaYbHq}.
So a general term in \(\sR\) might look like
\[
    \K^r\X^m\Y^t \otimes \K^s\X^u\Y^n
      = e^{\frac{h}{4}\big(r(\H \otimes 1) + s(1 \otimes \H)\big)} \X^m\Y^t \otimes \X^u\Y^n
      \qquad\text{for \(r,s \in \bbZ\) and \(m,t,u,n \in \bbZ_{\geq 0}\)}.
\]

Condition~\ref{D:QTHA.i} of Definition~\ref{D:QTHA} for \(a = \X\) reads
\((\K \otimes \X + \X \otimes \Kb) \cdot \sR = \sR \cdot (\X \otimes \K + \Kb \otimes \X)\).
This suggests that an asymmetric change in powers of \(\K\) and \(\Kb\) in the
general term of \(\sR\) needs to be related to one another by straightening.
Ultimately, we require a device which can introduce additional powers of
\(\K \otimes 1\) and \(1 \otimes \K\) via straightening.
A solution for this is to include powers of \(e^{\frac{h}{4}\H \otimes \H}\),
as straightening with expressions involving \(\X\) and \(\Y\) creates terms of
the form
\(e^{\frac{h}{4}(\H + a) \otimes (\H + b)}\)
through Lemma~\ref{L:calcUhSL2}\eqref{L:calcUhSL2.ii}.

This argument leads to
\begin{equation}\label{e:Ansatz1}
  \sR = \sum_{k,m,n,r,s,t,u} a_{k,r,s,m,n,t,u}(q) \,  e^{\frac{h}{4} \left(k\, \H \ot \H  + r\,\H \ot1 + s\,1\ot \H   \right)} \X ^m \Y ^t \ot \X ^u \Y ^n.
\end{equation}
as a conjectural form for \(\sR\), where \(a_{k,r,s,m,n,t,u}(q)\) is a function
of \(q\).

We begin by considering an even simpler form in which \(\X\) and \(\Y\) occur
exclusively in the first and second tensor components, respectively.
Namely, set \(t = 0\) and \(u = 0\) in~\eqref{e:Ansatz1} and write
\(a_{k,r,s,m,n}(q) \coloneqq a_{k,r,s,m,n,0,0}(q)\) to obtain the ansatz
\begin{equation}\label{e:Ansatz3}
  \sR = \sum_{k,m,n,r,s} a_{k,r,s,m,n}(q) \,  e^{\frac{h}{4} k  \left(\H \ot \H \right)} \K ^r \X ^m  \ot  \K ^s \Y ^n.
\end{equation}
For brevity, explicit mention of the dependence of \(a_{k,r,s,m,n}(q)\) on
\(q\) will henceforth be suppressed.

\subsubsection{Condition~\ref{D:QTHA.i} of Definition~\ref{D:QTHA}}
Since \(\De\) is an algebra morphism it is sufficient to show that this
Condition holds for the generators \(\H\), \(\X\) and \(\Y\).

\myss{For the generator \(\H\)}
The Condition asserts that
\((\H \otimes 1 + 1 \otimes \H ) \cdot \sR = \sR\cdot(\H \otimes 1 + 1 \otimes \H)\).
Since \(e^{\frac{h}{4} \left(k\, \H \otimes \H + r\,\H \otimes1 + s\,1\otimes \H \right)}\)
commutes with \(\H \ot 1 + 1\ot \H\), the Condition is equivalent to the termwise
equality
\[
  (\H\otimes 1 + 1\otimes\H )(\K^r\X^m \otimes \K^s\Y^n)
    =
  (\K^r\X^m  \otimes \K^s\Y^n) (\H \otimes 1 + 1\otimes \H).
\]
By Lemma~\ref{L:calcUhSL2}\eqref{L:calcUhSL2.ii}, the right hand side is
\[ (\K^r\X^m \otimes \K^s\Y^n)(\H \otimes 1 + 1\otimes \H ) =
    (\H \otimes 1 + 1 \otimes \H )(\K^r\X^m \otimes \K^s\Y^n)
      - 2(m-n) (\K^r\X^m \otimes \K^s\Y^n)
\]
so the Condition implies that \(m = n\).
Setting \(a_{k,r,s,n} \coloneqq a_{k,r,s,n,n}\),~\eqref{e:Ansatz3} simplifies to
\begin{equation}\label{e:Ansatz4}
  \sR = \sum_{k,n,r,s} a_{k,r,s,n} \, e^{\frac{h}{4}k (\H \ot \H )} (\K ^r \X ^n  \ot  \K ^s \Y ^n).
\end{equation}

\myss{For the generator \(\X\)}
By Lemma~\ref{L:calcUhSL2}\eqref{L:calcUhSL2.ii},
\begin{align*}
  (\tau\oc\De)(\X ) \cdot \sR & =
    \sum_{k,n,r,s} a_{k,r,s,n} e^{\frac{h}{4}k (\H \ot \H)}
      \big(\Kb^{2k-1-r}\X^n \ot \X\K^s\Y^n + \X\K^r\X^n\ot \Kb^{2k+1-s}\Y^n\big), \\
  \sR \cdot \De(\X) & =
    \sum_{k,n,r,s} a_{k,r,s,n} \, e^{\frac{h}{4}k(\H\ot\H)}
      \big(\K^r\X^{n+1} \ot \K^s\Y^n\K + \K^r\X^n\Kb \ot \K^s\Y^n\X\big).
\end{align*}
The equation \((\tau\oc\De)(\X) \cdot \sR = \sR \cdot \De(\X)\) may be
rearranged as
\begin{multline}\label{eq:AeqB}
  A \coloneqq
    \sum_{k,n,r,s}
      a_{k,r,s,n} \,  e^{\frac{h}{4}k (\H \ot \H)}
        \Big(\big(\Kb^{2k-1-r}\X^n \ot \X\K^s\Y^n\big) - \big(\K^r\X^n\Kb \ot \K^s\Y^n\X\big)\Big) \\
    =
    \sum_{k,n,r,s}
      a_{k,r,s,n} \,  e^{\frac{h}{4}k (\H \ot \H)}
        \Big(\big(\K^{r}\X^{n+1} \ot \K^s\Y^n\K\big) - \big(\X\K^r\X^n \ot \Kb^{2k+1-s}\Y^n\big)\Big)
    \eqqcolon
  B.
\end{multline}

To simplify \(A\), note that \(\X\,\K^s  = \qb^s \K^s\X\) and
\(\X^n \Kb = q^n \Kb\X^n\) from Lemma~\ref{L:calcUhSL2}\eqref{L:calcUhSL2.iii},
so
\[
  A = \sum_{k,n,r,s}
    a_{k,r,s,n} \,  e^{\frac{h}{4}k (\H \ot \H)}
      \Big(\qb^s\big(\Kb^{2k-1-r}\X^n \ot \K^s\X\Y^n\big) -  q^n\big(\K^{r-1}\X^n \ot \K^s\Y^n\X\big)\Big).
\]
Observe that if we had \(s = -n\) and \(k = 1\), the commutator \([\X,\Y^n]\)
appears from the second tensor factors of this expression.
Doing so, and writing \(a_{r,n}\) for \(a_{1,r,-n,n}\),
Lemma~\ref{L:calcUhSL2}\eqref{L:calcUhSL2.v} allows us to compute the
commutator to yield
\begin{equation}\label{e:Aform2}
  A = e^{\frac{h}{4}\H \ot \H} \sum_{n,r} a_{r,n+1} [n+1]\, q^{n+1} \K^{r-1}\X^{n+1} \ot [\H + n] \Kb^{n+1}\Y^{n}.
\end{equation}

To simplify \(B\), set \(s = -n\) and \(k = 1\) as above, and
apply Lemma~\ref{L:calcUhSL2}\eqref{L:calcUhSL2.iii} to obtain
\[
  B =
    e^{\frac{h}{4}\H \ot \H}
    \sum_{n,r} a_{r,n}(q) \K^r\X^{n+1} \otimes \Big(q^n \Kb^{n-1} \Y^n -\qb^r\Kb^{n+3}\Y^n\Big).
\]
The second term simplifies if we had \(r = n\).
Again, doing so, and writing \(a_n\) for \(a_{n,n}\), yields
\[
  B =
    e^{\frac{h}{4}\H \ot \H}
    \sum_{n \geq 0} a_{n}\,(q-\qb) \K^n\X^{n+1} \otimes [\H+n]\,\Kb^{n+1}\Y^n.
\]
Also setting \(r = n\) for \(A\) in~\eqref{e:Aform2}, we have
\[
  A =
    e^{\frac{h}{4}\H \ot \H}
    \sum_{n\geq0} a_{n+1} [n+1]\, q^{n+1} \K^n\X^{n+1} \otimes [\H+n] \Kb^{n+1}\Y^n.
\]

Comparing \(A\) and \(B\) shows that the \(a_n = a_n(q)\) must satisfy
the two-term recurrence equation
\[
  a_{n+1}(q) \cdot [n+1] q^{n+1}= a_n(q) \cdot (q - \qb)
    \qquad\text{with initial condition \(a_0(q)=1\)},
\]
whence \(a_n  = \frac{(q - \qb)^n}{[n]!} \qb^{\frac{1}{2}n(n+1)}\).
Substituting these settings into~\eqref{e:Ansatz4} gives
\begin{equation*}
  \sR = e^{\frac{h}{4}\H \ot \H} \sum_{n \ge 0} \frac{(q - \qb)^n}{[n]! } \qb^{\frac{1}{2}n(n+1)}\,(\K^n\X^n  \ot \Kb^n\Y^n).
\end{equation*}
Now \(\K^n\X^n = q^{\frac{1}{2}n(n-1)}(\K\X)^n\) and
\(\Kb^n\Y^n = q^{\frac{1}{2}n(n-1)}(\Kb\Y)^n\) from
Lemma~\ref{L:calcUhSL2}\eqref{L:calcUhSL2.iv}, so
\begin{equation}\label{e:Ansatz5}
  \sR = e^{\frac{h}{4}\H \ot \H} \sum_{n\ge0}  \frac{(q - \qb)^n}{[n]! } q^{\frac{1}{2}n(n-3)}\,  (\K\X)^n \ot (\Kb\Y)^n.
\end{equation}

\myss{For the generator \(\Y\)}
It is readily shown that this condition is satisfied by the expression for
\(\sR\) given in~\eqref{e:Ansatz5}.

\subsubsection{Condition~\ref{D:QTHA.ii} of Definition~\ref{D:QTHA}}\label{SS:q-exp}
It is immediate from the definition of the quantum exponential function
in~\eqref{e:qExpFn} and the expression for \(\sR\) given in~\eqref{e:Ansatz5}
that
\begin{equation}\label{e:Ras:qExpFn}
  \sR =
    e^{\frac{h}{4}\H\ot\H}\,
    \Exp{q}\big(\lam_q \K\X\ot\Kb\Y\big)
  \qquad\mbox{where $\lam_q \coloneqq \qb (q-\qb)$}.
\end{equation}
Condition~\ref{D:QTHA.ii} of Definition~\ref{D:QTHA} requires that
\((\De\ot\id)(\sR) = \sR_{13}\cdot\sR_{23}\).  It is readily seen that
\begin{equation}\label{e:DeIdR}
  (\De\ot\id)(\sR) =
    e^{\frac{h}{4} (\H \otimes 1 \otimes \H + 1\otimes\H \otimes\H)}\,
    \Exp{q}\big(\lam_q (\K\X\otimes\K^2 + 1\otimes\K\X)\otimes\Kb\Y\big)
\end{equation}
since \(\De\) is an algebra morphism, and \(\De(\K) = \K \ot \K\) from~\S\ref{S:DSeK}.

On the other hand, recalling notation from Definition~\ref{D:QTHA}, we have
\begin{align*}
\sR_{13}\cdot\sR_{23} & =
  \Big(e^{\frac{h}{4}\,\H\ot1\ot\H}\,\Exp{q}\big(\lam_q (\K\X\ot1\ot\Kb\Y)\big)\Big)
  \Big(e^{\frac{h}{4}\,1\ot\H\ot\H}\,\Exp{q}\big(\lam_q (1\ot\K\X\ot\Kb\Y)\big)\Big) \\
  & =
    e^{\frac{h}{4}(\H \ot1\ot \H  + 1\ot \H \ot \H )}
      \Exp{q}\big(\lam_q (\K\X \ot \K^2 \ot \Kb\Y)\big) \cdot
      \Exp{q}\big(\lam_q (1 \ot \K\X \ot \Kb\Y)\big).
\end{align*}
where we have used Lemma~\ref{L:1otX} to straighten an exponential.
Let
\(C \coloneqq \K\X \ot \K^2 \ot \Kb\Y\) and
\(D \coloneqq 1\ot \K\X\ot\Kb\Y\).
Then Lemma~\ref{L:calcUhSL2}\eqref{L:calcUhSL2.iii} gives
\(DC = \qb^2\big(\K\X \ot \K^3\X \ot (\Kb\Y )^2\big) = \qb^2 CD\), so
Lemma~\ref{L:qExpFact} applies to give
\[
  \sR_{13}\cdot\sR_{23} =
    e^{\frac{h}{4}(\H \ot 1 \ot \H + 1 \ot \H \ot \H)} \cdot
      \Exp{q}\big(\lam_q (\K\X \ot \K^2 \ot \Kb\Y + 1 \ot \K\X \ot \Kb\Y)\big).
\]
Comparing with~\eqref{e:DeIdR} shows that this is \((\De\ot\id)(\sR)\), so the
Condition is satisfied.

\subsubsection{Condition~\ref{D:QTHA.iii} of Definition~\ref{D:QTHA}}
It may be shown similarly that this Condition is also satisfied.

We have therefore proven the following:
\begin{theorem}\label{T:Rmorphism1}
  A universal \(\sR\)-matrix for \(\Uhsltwo\) is given by
  \[ \sR = e^{\frac{h}{4} \H\ot\H} \sum_{n\ge0} \frac{(q - \qb)^n}{[n]!} q^{\frac{1}{2}n(n-3)}\,(\K\X)^n\ot(\Kb\Y)^n. \]
\end{theorem}

\begin{remark}\label{R:Observe}
  Two comments on the derivation of \(\sR\):
\begin{enumerate}
  \item[(A)]\label{R:Observe.A} A more detailed calculation, dealing with the
    complete form for \(\sR\) proposed in~\eqref{e:Ansatz1}, will show that
    the \(\sR\)-matrix found in Theorem~\ref{T:Rmorphism1} is the only one
    of this form.

  \item[(B)]\label{R:Observe.B} The salient aspects of the derivation of \(\sR\)
    are \textbf{(i)}~the appearance of the commutator \([\X,\Y^n]\)
    in~(\ref{e:Aform2}), and \textbf{(ii)}~the fact that \(\sR\) may be expressed
    succinctly in terms of the quantum exponential function
    in~\eqref{e:Ras:qExpFn}.
    The factorization property given in Lemma~\ref{L:qExpFact} for the quantum
    exponental is crucial in showing that Conditions~\ref{D:QTHA.ii}
    and~\ref{D:QTHA.iii} of Definition~\ref{D:QTHA} hold.
\end{enumerate}
\end{remark}

\section{Ribbon Hopf algebra structure on \texorpdfstring{$\Uhsltwo$}{Uhsl2}}\label{S:ribbon}
In this Section, we use our straightening methods to compute a ribbon Hopf
algebra structure on \(\Uhsltwo\).
These calculations are similar to those in~\cite[Appendix A]{oh}, but are
simpler and more direct.
For instance, our calculation of \(\sS(\bu)\), based on the combinatorial
identity in Lemma~\ref{L:CauchyExp}, is significantly shorter than the
corresponding calculation in~\cite{oh}.

To begin, let \(\cA\) be a quasi-triangular Hopf algebra with \(\sR\)-matrix
\(\sR = \sum_i \alpha_i \otimes \beta_i\).
Drinfel'd observed in~\cite{d2} that the element
\begin{equation}\label{eq:foru.1}
  \bu\coloneqq \sum\nolimits_i  \sS(\be_i)\cdot\al_i \in \cA.
\end{equation}
satisfies the following properties:
\begin{enumerate}
  \item\label{L:uProps.i}
    For all \(x \in \cA\), \(\sS^2(x) = \bu\, x\, \bu^{-1}\);
  \item\label{L:uProps.ii}
    \(\Delta(\bu) = (\tau(\sR)\cdot\sR)^{-1}(\bu\otimes\bu) = (\bu\otimes\bu)(\tau(\sR)\cdot\sR)^{-1}\); and
  \item\label{L:uProps.iii}
    \(\bu\) is invertible, with
    \(\bu^{-1} = \sum\nolimits_i \sS^{-1}(\overline\beta_i)\overline\alpha_i\), where
    \(R^{-1} = \sum\nolimits_i \overline\alpha_i\otimes\overline\beta_i\).
\end{enumerate}
See~\cite[\emph{pp.}180--184]{k} for details.

Ribbon Hopf algebras are quasi-triangular Hopf algebras which admit a sort of
square root to the element \(\bu\).
They were introduced by Reshetikhin and Turaev~\cite{rt2} in order
to construct a polynomial invariant for framed links.

\begin{definition}\label{D:RHA}
A \emph{ribbon Hopf Algebra} is quasi-triangular Hopf algebra \((\cA,\sR)\)
equipped with an element \(\bv \in \cA\), called a \emph{ribbon element}, satisfying:
\begin{enumerate}
  \item \label{D:RHA.i}   \(\bv\) is central;
  \item \label{D:RHA.ii}  \(\bv^2= \sS(\bu)\cdot \bu\);
  \item \label{D:RHA.iii} \(\sS(\bv) = \bv\);
  \item \label{D:RHA.iv}  \(\varep(\bv)=1\); and
  \item \label{D:RHA.v}   \(\De(\bv) = (\tau(\sR)\cdot\sR)^{-1} \cdot (\bv\ot\bv)\).
\end{enumerate}
\end{definition}

We now explicitly compute a ribbon element for \(\Uhsl{2}\) by first
directly computing the element \(\bu\) for the \(\sR\)-matrix in
Theorem~\ref{T:Rmorphism1} from~\eqref{eq:foru.1}, and then obtaining \(\bv\)
from Condition~\ref{D:RHA.ii} of Definition~\ref{D:RHA}.

\begin{lemma}\label{L:RHA:u}
  \(\displaystyle \bu = e^{-\frac{h}{4} \H ^2}\sum_{n\ge0} (-1)^n \frac{(q-\qb)^n}{[n]!}  \: \qb^{\frac{1}{2}n(n+3)} \: \Kb^{2n} \Y ^n \X ^n\).
\end{lemma}
\begin{proof}
From Theorem~\ref{T:Rmorphism1} and
Lemma~\ref{L:calcUhSL2}\eqref{L:calcUhSL2.iv},
\[
  \sR = \sum_{m,n\geq0} \frac{h^m}{4^m m!} \, c_n \,\qb^{n(n-1)}
    \big(\H^m\K^n\X^n\big) \otimes \big(\H^m\Kb^n\Y^n\big)
    \quad\text{where}\;
    c_n\coloneqq\frac{(q-\qb)^n}{[n]!} \qb^{\frac{1}{2}n(n-3)}
\]
and so, from~\eqref{eq:foru.1},
\begin{equation*}
  \bu = \sum_{m,n\geq0} \frac{h^m}{4^m m!} \,
    c_n\,\qb^{n(n-1)}\,\big(\sS(\H^m\Kb^n\Y^n)\big)\cdot\big(\H^m\K^n\X^n\big).
\end{equation*}
To straighten the summand, note that \(\sS(\Kb) = \K\) from \S\ref{S:DSeK},
and that \(\H\) and \(\K\) commute, so
\[
  \big(\sS(\H ^m \Kb^n \Y ^n)\big) \cdot \big(\H ^m \K ^n \X ^n\big)
    = (-1)^{m+n} \qb^n \Y ^n \H ^{2m}\K ^{2n} \X ^n
    = (-1)^{m+n} \qb^n q^{2n^2} (\H +2n)^{2m} \K ^{2n}\Y ^n \X ^n
\]
from Lemma~\ref{L:calcUhSL2}\eqref{L:calcUhSL2.ii} and~\eqref{L:calcUhSL2.iii}.
So
\[ \bu = \sum_{n\ge0} (-1)^n c_n \, q^{n^2} e^{-\frac{h}{4}(\H +2n)^2} \K ^{2n}\Y ^n \X ^n. \]
The result then follows by observing that
\(e^{-\frac{h}{4}(\H +2n)^2} = \qb^{2n^2} e^{-\frac{h}{4} \H ^2} \cdot \Kb^{4n}\).
\end{proof}

\begin{lemma}\label{L:SuKb4}
  \(\sS(\bu) = \Kb^4 \bu\).
\end{lemma}
\begin{proof}
Let
$d_n \coloneqq (-1)^n \frac{(q-\qb)^n}{[n]!} \, \qb^{\frac{1}{2}n(n+3)}$.
Then
\[ \Kb^4\bu = e^{-\frac{h}{4} \H^2} \sum_{n \geq 0} d_n\, \Kb^{2(n+2)}\Y^n\X^n. \]
From the definition of \(\sS\) at the beginning of \S\ref{S:Rmorphism} and from
its action on \(\K\) given in \S\ref{S:DSeK},
\[ \sS(\bu) = \sum_{n \geq 0} d_n \, \X^n\Y^n\K^{2n}e^{-\frac{h}{4} \H^2}. \]
Now \(\K^{2n}\) and \(e^{-\frac{h}{4} \H^2}\)
commute and, by Lemma~\ref{L:calcUhSL2}\eqref{L:calcUhSL2.ii}, both of
these commute with \(\X^n \Y^n\), so
\begin{equation*}
  \sS(\bu) = e^{-\frac{h}{4} \H ^2} \sum_{n\ge0} d_n \, \K^{2n}\X^n\Y^n.
\end{equation*}
Using the Poincar\'e--Birkhoff--Witt Theorem, we can establish the equality of
\(\Kb^4\bu\) and \(\sS(\bu)\) by comparing coefficients.
First we compare the coefficients of \(\Y^n\X^n\) with the aid of
Lemma~\ref{L:XaYbHq}.
Doing so, the assertion of the Lemma is therefore equivalent to the identity
\begin{equation*}
d_n \, \Kb^{2(n+2)} =
  \sum_{m\ge n} d_m  \, \frac{[m]!^2}{[n]!^2} \, \K ^{2m} \, \Qbinom{\H }{m-n}
\end{equation*}
and so, by changing the index of summation to \(s\coloneqq m-n\), to
showing that
\[
\Kb^{4(n+1)} = A_n
 \qquad\mbox{where}\qquad
A_n \coloneqq \sum_{s\ge 0} \frac{d_{n+s}}{d_n}  \, \frac{[n+s]!^2}{[n]!^2} \, \K ^{2s} \, \Qbinom{\H }{s}.
\]
It is in this form that we shall prove the lemma.
From~\eqref{e:NTN:qKqb} and~\eqref{e:BinomH},
\[
\Qbinom{\H}{s} = \frac{1}{[s]!} \, \prod_{r=0}^{s-1}
\frac{\qb^r \K ^2 - q^r \Kb^2}{q-\qb}
= \frac{1}{[s]!} \, \frac{\qb^{\frac{1}{2}s(s-1)}}{(q-\qb)^s} \, \Kb^{2s} \, \prod_{r=0}^{s-1}\big(\K^4-q^{2r}\big)
\]
so
\[
  A_n =  \sum_{s\ge0} (-1)^s \; \qb^{s(n+s+1)}\Qbinom{n+s}{s} \,\prod_{r=0}^{s-1}\big(\K^4-q^{2r}\big).
\]
We shall now transform the quantum binomial coefficient into a \(q\)-binomial
coefficient through~\eqref{e:Qbin:CombBin} and then to a negative \(q\)-binomial
coefficient to obtain
\[
  A_n
  = \sum_{s \geq 0} (-1)^s \qb^{s(2n+s+1)} \binom{n+s}{s}_{q^2} \, \prod_{r=0}^{s-1} \big(\K^4-q^{2r}\big)
  = \sum_{s \geq 0} \binom{-(n+1)}{s}_{q^2} \, \prod_{r=0}^{s-1} \big(\K ^4 - q^{2r}\big).
\]
Since \(q^2 = e^{h}\) and \(\K^4 = e^{h\H}\),
Lemma~\ref{L:CauchyExp} gives \(A_n = e^{-(n+1)h\H} = \Kb^{4(n+1)}\).
\end{proof}

To find a candidate element for \(\bv\), observe that
Lemma~\ref{L:SuKb4} simplifies Definition~\ref{D:RHA}\ref{D:RHA.ii} to
\(\bv^2 = \sS(\bu) \cdot \bu = \Kb^4 \bu^2\).
But property~\ref{L:uProps.i} of \(\bu\) together with \S\ref{S:DSeK} says
\(\bu \Kb= \sS^2(\Kb)\bu = \Kb\bu\).
Thus \(\Kb^2\bu\) is a putative expression for the ribbon
element, we now show satisfies the remaining conditions of Definition~\ref{D:RHA}.

\begin{theorem}\label{L:RHA:v}
  A ribbon element for \((\Uhsltwo,\sR)\) is
  \begin{align*}
    \bv \coloneqq \Kb^2\bu = e^{-\frac{h}{4} \H ^2} \cdot \sum_{n\ge0} \frac{(\qb-q)^n}{[n]!} \, \qb^{\frac{1}{2}n(n+3)}  \Kb^{2(n+1)} \Y^n\X^n.
  \end{align*}
\end{theorem}
\begin{proof}
We verify Conditions (i)--(v) of Definition~\ref{D:RHA}.

Condition~\ref{D:RHA.i} asks for \(\bv\) to be central.
It is enough to show \(\bv\) commutes with the generators \(\H\), \(\X\), and \(\Y\).
Let \(T_n \coloneqq e^{-\frac{h}{4}\H^2}\Kb^{2(n+1)}\Y^n\X^n\) be a general
term of \(\bv\).
Lemma~\ref{L:calcUhSL2}\eqref{L:calcUhSL2.ii} shows that \(\H\) commutes with
\(T_n\), and hence \(\bv\).
That \(\X\) and \(\Y\) commute with \(\bv\) follow from similar calculations,
so we only show commutation with \(\X\).
From Lemma~\ref{L:calcUhSL2}\eqref{L:calcUhSL2.ii},
\(\X  e^{-\frac{h}{4}\H^2} = \qb^2 e^{-\frac{h}{4}\H^2}\K^4\X\).
Thus,
\begin{align*}
\X  \,T_n
  &= e^{-\frac{h}{4} \H ^2} q^{2n} \Kb^{2(n-1)} (\X \,\Y ^n) \X ^n
    &\explain{Lemma~\ref{L:calcUhSL2}\eqref{L:calcUhSL2.iii}}	\\
  &= e^{-\frac{h}{4} \H ^2} q^{2n} \Kb^{2(n-1)}
  \big( \Y ^n \X ^{n+1} + [n]\, [\H +n-1]\, \Y ^{n-1} \X ^n\big)
    &\explain{Lemma~\ref{L:calcUhSL2}\eqref{L:calcUhSL2.v}}
\end{align*}
and so, using the notation \(c_n\) from the proof of Lemma~\ref{L:RHA:u},
\begin{align*}
\X  \bv
&= e^{-\frac{h}{4} \H ^2} \sum_{n\ge0} c_n\,  q^{2n} \Kb^{2(n-1)}
\big(  \Y ^n \X ^{n+1} + [n]\, [\H +n-1]\, \Y ^{n-1} \X ^n\big)	\\
& =   e^{-\frac{h}{4} \H ^2} \sum_{n\ge1} q^{2n-2} \Kb^{2(n-2)}
\big( c_{n-1}  + c_n \, q^2 \Kb^2 [n]\, [\H +n-1]\,\big) \Y ^{n-1} \X ^n
\end{align*}
by shifting the summation index for the term containing \(\Y^n\X^{n+1}\) by
one. The bracketed term is equal to
\(\frac{(\qb-q)^{n-1}}{[n-1]!} \,\qb^{n-1} \Kb^4  \qb^{\frac{1}{2}(n^2+3n-4)}\).
Therefore
\[ \X \bv =
  e^{-\frac{h}{4} \H ^2} \sum_{n\ge1} \qb^{\frac{1}{2}(n^2+n-2)}
  \frac{(\qb-q)^{n-1}}{[n-1]!} \,\Kb^{2n} \Y ^{n-1} \X ^n.
\]
The right hand side is readily seen to be equal to \(\bv\X\) by  shifting the
summation index for the expression for \(\bv\) to start at \(0\).

Condition~\ref{D:RHA.ii} holds by construction of \(\bv\).
Conditions~\ref{D:RHA.iii} and~\ref{D:RHA.iv} are immediate by definition
of the maps \(\sS\) and \(\varep\) for \(\Uhsl{2}\).
Condition~\ref{D:RHA.v} amounts to showing
\((\tau(\sR)\cdot\sR)\cdot(\Kb^2\otimes\Kb^2)\cdot\Delta(\bu) = (\Kb^2\otimes\Kb^2)\cdot(\bu\otimes\bu)\).
Since each term of \(\sR\), and thus \(\tau(\sR)\), commutes with
\(\Kb^2\otimes\Kb^2\), it is sufficient to show
\((\tau(\sR)\cdot\sR)\cdot\Delta(\bu) = \bu\otimes\bu\).
But this is property~\ref{L:uProps.ii} of \(\bu\).
\end{proof}

\section{Straightening in \texorpdfstring{$\Uhsl{n+1}$}{Uhsln}}\label{S:Uhsln}
In this Section, we generalize the methods of \S\ref{S:QUEA} from \(\Uhsl{2}\)
to \(\Uhsl{n+1}\) for \(n \geq 2\).
These methods are applied in \S\ref{S:Uhsl3-coeffs} to construct an
\(\sR\)-matrix for \(\Uhsl{n+1}\).
We begin with a summary on the quantized universal enveloping algebra of
\(\fsl_{n+1}\) for all \(n \geq 1\); for details, see, for example,~\cite{cp}.
By definition, \(\Uhsl{n+1}\) is the
\(\bbC[\![h]\!]\)-algebra generated by \(n\) \(\fsl_2\) triples
\((\X_i,\H_i,\Y_i)\), \(i = 1,\ldots,n\), subject to the relations
\([\H_i,\H_j] = 0\),
\begin{align*}
  [\H_i,\X_j] & = \begin{dcases} \hphantom{-\X_j}\mathllap{2\X_i}  & j = i, \\ -\X_j & j = i \pm 1, \\ \hphantom{-\X_j}\mathllap{0} & \text{otherwise}, \end{dcases} &
  [\H_i,\Y_j] & = \begin{dcases} -2\Y_i & j = i, \\ \hphantom{-2\Y_i}\mathllap{\Y_j}  & j = i \pm 1, \\ \hphantom{-2\Y_i}\mathllap{0} & \text{otherwise}, \end{dcases} &
  [\X_i,\Y_j] & = \delta_{ij} \frac{\K_i^2 - \Kb_i^2}{q - \qb},
\end{align*}
where \(\K_i \coloneqq e^{\frac{h}{4}\H_i}\), and for each \(i,j\) with \(|i -
  j| = 1\), the \emph{\(\q\)-Serre relations}
\begin{align*}
  \X_i^2\X_j - (q + \qb)\X_i\X_j\X_i + \X_j\X_i^2 & = 0, &
  \Y_i^2\Y_j - (q + \qb)\Y_i\Y_j\Y_i + \Y_j\Y_i^2 & = 0,
\end{align*}
A Hopf algebra structure on \(\Uhsl{n+1}\) may be defined on the generators
\((\X_i,\H_i,\Y_i)\) by taking the maps from the beginning of
\S\ref{S:Rmorphism} on each \(\fsl_2\)-triple.

\subsection{Simplifying the Serre relations}\label{SS:Uhsl3-serre}
As indicated in~\eqref{D:QhUSL.rels}, straightening methods amount to
interpreting quadratic relations as commutation relations.
However, the \(q\)-Serre relations in the definition of \(\Uhsl{n+1}\) are
cubic.
In order to build a straightening framework to deal with them, we must somehow
break them into quadratic relations.
One way to do so is to rewrite the \(q\)-Serre relation between \(\X_i\)
and \(\X_{i+1}\) as
\begin{align*}
  \X_i^2\X_{i+1} - (q + \qb)\X_i\X_{i+1}\X_i + \X_{i+1}\X_i^2
    & = \X_i(\X_i\X_{i+1} - \qb\X_{i+1}\X_i) - (q\X_i\X_{i+1} - \X_{i+1}\X_i)\X_i \\
    & = \qbr\X_i(\qr\X_i\X_{i+1} - \qbr\X_{i+1}\X_i) - \qr(\qr\X_i\X_{i+1} - \qbr\X_{i+1}\X_i)\X_i = 0.
\end{align*}
Similarly, the \(q\)-Serre relation between \(\X_{i+1}\) and \(\X_ii\) can be
rewritten as
\begin{equation*}
  \qbr(\qr\X_i\X_{i+1} - \qbr\X_{i+1}\X_i)\X_{i+1} - \qr\X_{i+1}(\qr\X_i\X_{i+1} - \qbr\X_{i+1}\X_i) = 0.
\end{equation*}
So, setting \(\X_{i,i+1} \coloneqq \qr\X_i\X_{i+1} - \qbr\X_{i+1}\X_i\)
the \(q\)-Serre relations involving \(i\) and \(i+1\) can be written as quadratic
relations
\[
  \qbr\X_i\X_{i,i+1} - \qr\X_{i,i+1}\X_i = 0,
    \quad\text{and}\quad
  \qbr\X_{i,i+1}\X_{i+1} - \qr\X_{i+1}\X_{i,i+1} = 0.
\]

\begin{example}\label{ex:Uhsl3}
When \(n = 2\), the elements
\(\H_1 \prec \H_2 \prec \Y_1 \prec \Y_{12} \prec \Y_2 \prec \X_1 \prec \X_{12} \prec \X_2\)
of \(\Uhsl{3}\) form an ordered set of algebra generators in which all
relations are quadratic relations amongst pairs of generators.
Moreover, the set of generators gives rise to a Poincar\'e--Birkhoff--Witt basis
\[
  \cB \coloneqq
    \big\{\H_1^{r_1}\H_2^{r_2}\Y_1^{s_1}\Y_{12}^{s_{12}}\Y_2^{s_2}\X_1^{t_1}\X_{12}^{t_{12}}\X_2^{t_2} : r_i, s_i, t_i \in \bbZ_{\geq 0}\big\}
\]
and straightening can be used  effectively to perform computations.
\end{example}

\subsection{Higher Degree Generators in \texorpdfstring{$\Uhsl{n+1}$}{Uhsln+1}}
When \(n \geq 3\), the \(\X_{i,i+1}\) are not enough to transform all
relations of \(\Uhsl{n+1}\) into commutation relations.
For example, the \(q\)-Serre relations yield cubic relations between
\(\X_{1,2}\) and \(\X_3\).
Iterating the reasoning of \S\ref{SS:Uhsl3-serre}, we see that we should
inductively define elements \(\X_{i,j}\) for each pair of indices
\(1 \leq i \leq j \leq n\), as follows:
for each \(i = 1,\ldots,n\), set \(\X_{i,i} \coloneqq \X_i\);
and for each pair of indices \(1 \leq i < j \leq n\), inductively define
\begin{equation}\label{D:Xij}
  \X_{i,j} \coloneqq \qr\X_i\X_{i+1,j} - \qbr\X_{i+1,j}\X_i.
\end{equation}
A short induction argument shows that the ideal of relations in \(\Uhsl{n+1}\)
are generated by quadratic relations amongst the algebra generators
\begin{equation}\label{e:prec-order}
  \H_1 \prec \cdots \H_2 \prec \cdots \prec \H_n \prec \Y_1 \prec \Y_{12} \prec \cdots \prec \Y_{1n} \prec \Y_2
    \prec \cdots \prec \Y_{n-1,n} \prec \Y_n \prec \X_1 \prec \X_{12} \prec \cdots \prec \X_{n-1,n} \prec \X_n.
\end{equation}
Moreover, the ordering above yields a Poincar\'e--Birkhoff--Witt basis as in
Example~\ref{ex:Uhsl3}.
This final statement can be established directly as in~\cite{r} or via the
general theory of quantum groups, as developed in~\cite[Chapter 40]{lus93}.

The elements \(\X_{i,j}\) have a more symmetric, and rather useful, description:

\begin{lemma}[Splitting]\label{L:splitting}
  Let \(1 \leq i < j \leq n\) and let \(s \in \{i,\ldots,j-1\}\).
  Then
  \[ \X_{i,j} = \qr\X_{i,s}\X_{s+1,j} - \qbr\X_{s+1,j}\X_{i,s}. \]
\end{lemma}

To prove this, we give a formula for \(\X_{i,j}\) in terms of monomials in the
\(\X_i,\ldots,\X_j\) indexed by orientations on a path of length \(j - i\).
More precisely, let \(P_n\) denote the Dynkin diagram for \(\fsl_{n+1}\),
i.e. a path consisting of \(n\) vertices labelled from left to right by
the integers \(1\) to \(n\).
For \(1 \leq i < j \leq n\), let \(P_{i,j}\) denote the induced subgraph
of \(P\) obtained by taking the vertices labelled \(i,i+1,\ldots,j\).
Let \(\cD_{i,j}\) denote the set of orientations on \(P_{i,j}\).
For each orientation \(D \in \cD_{i,j}\), let
\begin{equation*}
  D_\rightarrow \coloneqq \#\{\ell \rightarrow \ell + 1 \in D\},
    \quad\text{and}\quad
  D_\leftarrow  \coloneqq \#\{\ell \leftarrow \ell + 1 \in D\}
\end{equation*}
be the number of right- and left-pointing arrows in the orientation \(D\).
For each \(D \in \cD_{i,j}\), set
\begin{equation*}
  q^D \coloneqq (-1)^{D_\leftarrow} q^{\frac{1}{2}(D_\rightarrow - D_\leftarrow)}.
\end{equation*}
For \(D \in \cD_{i,j}\), let \(\X_D\) be the monomial \(\X_i,\ldots,\X_j\)
constructed as follows:
Begin by writing \(\X_j\).
Next, if \(j - 1 \to j \in D\), then place \(\X_{j-1}\) to the \emph{left} of \(\X_j\);
otherwise, \(j-1 \leftarrow j \in D\) and so place \(\X_{j-1}\) on the \emph{right} of \(\X_j\).
Next, if \(j-2 \rightarrow j - 1 \in D\), then place \(\X_{j-2}\) at the leftmost end;
otherwise, place it on the rightmost.
At each step, regard a right-pointing arrow \(\ell - 1 \rightarrow \ell\) as
indicating that \(\X_{\ell - 1}\) and \(\X_\ell\) are to be positioned ``in
order'', so that \(\X_{\ell - 1}\) appears before \(\X_\ell\);
similarly, a left-pointing arrow \(\ell - 1 \leftarrow \ell\) indicates that
\(\X_{\ell - 1}\) and \(\X_\ell\) appear ``out of order''.
Continue this process until all of \(\X_j\) to \(\X_i\) have been placed and
call the result \(\X_D\).

\begin{example}
Suppose \(n = 8\), \(i = 2\) and \(j = 6\). Then
\[
  P_7 =
\begin{tikzpicture}[main node/.style={inner sep=0,minimum size =.12cm,circle,fill=black!80,draw},node distance = 0.75cm,scale=0.4]
  \node[main node] (1) [label=\(1\)] {};
  \node[main node] (2) [right of=1, label=\(2\)] {};
  \node[main node] (3) [right of=2, label=\(3\)] {};
  \node[main node] (4) [right of=3, label=\(4\)] {};
  \node[main node] (5) [right of=4, label=\(5\)] {};
  \node[main node] (6) [right of=5, label=\(6\)] {};
  \node[main node] (7) [right of=6, label=\(7\)] {};
  \draw (1) -- (2) -- (3) -- (4) -- (5) -- (6) -- (7);
\end{tikzpicture},
  \quad\text{and}\quad
  P_{2,6} =
\begin{tikzpicture}[main node/.style={inner sep=0,minimum size =.12cm,circle,fill=black!80,draw},node distance = 0.75cm,scale=0.4]
  \node[main node] (2) [right of=1, label=\(2\)] {};
  \node[main node] (3) [right of=2, label=\(3\)] {};
  \node[main node] (4) [right of=3, label=\(4\)] {};
  \node[main node] (5) [right of=4, label=\(5\)] {};
  \node[main node] (6) [right of=5, label=\(6\)] {};
  \draw (2) -- (3) -- (4) -- (5) -- (6);
\end{tikzpicture}.
\]
The set \(\cD_{2,6}\) has \(2^4 = 16\) elements, one being
$
  D =
\begin{tikzpicture}[main node/.style={inner sep=0,minimum size =.12cm,circle,fill=black!80,draw},node distance = 0.75cm,scale=0.4]
  \node[main node] (2) [right of=1, label=\(2\)] {};
  \node[main node] (3) [right of=2, label=\(3\)] {};
  \node[main node] (4) [right of=3, label=\(4\)] {};
  \node[main node] (5) [right of=4, label=\(5\)] {};
  \node[main node] (6) [right of=5, label=\(6\)] {};
  \draw[->-] (2) -- (3);
  \draw[-<-] (3) -- (4);
  \draw[->-] (4) -- (5);
  \draw[->-] (5) -- (6);
\end{tikzpicture}.
$
Then \(D_\rightarrow = 3\), \(D_\leftarrow = 1\), \(q^D = -q\). The
construction of \(\X_D\) proceeds through the following steps: \(\X_6\),
\(\X_5\X_6\), \(\X_4\X_5\X_6\), \(\X_4\X_5\X_6\X_3\), and finally
\(\X_D = \X_2\X_4\X_5\X_6\X_3\).
\end{example}

\begin{lemma}\label{L:CombDesc.Xij}
  Let \(1 \leq i < j \leq n\).
  Then \(\X_{i,j} = \sum_{D \in \cD_{i,j}} q^D\X_D\).
\end{lemma}
\begin{proof}
  Proceed by induction on the difference \(j - i\).
  When \(j - i = 1\), the formula reduces to the definition of \(\X_{i,i+1}\)
  in~\eqref{D:Xij}.
  In general,
  \begin{align*}
    \X_{i,j}
      & = \qr\X_i\X_{i+1,j} - \qbr\X_{i+1,j}\X_i \\
      & = \sum_{D \in \cD_{i+1,j}} (-1)^{D_\leftarrow}q^{\frac{1}{2}((D_\rightarrow + 1)- D_\leftarrow)}\X_i\X_D + (-1)^{D_\leftarrow + 1}q^{\frac{1}{2}(D_\rightarrow - (D_\leftarrow + 1))}\X_D\X_i.
  \end{align*}
  Now observe that the first terms in the sum on the right hand side yield the
  monomials \(\X_{D'}\) with \(D' \in \cD_{i,j}\) such that the arrow between
  \(i\) and \(i + 1\) is right-pointing; likewise, the second terms in the sum
  correspond to \(D' \in \cD_{i,j}\) with arrows \(i \leftarrow i + 1\).
  Putting everything together gives the Lemma.
\end{proof}

\begin{remark}\label{R:good-reps}
Since \(\X_i\X_j = \X_j\X_i\) whenever \(\abs{i - j} > 1\), the \(\X_D\) are
simply convenient representatives of commutation equivalence class of monomials
in the \(\X_i\).
The essential information encoded by the orientation \(D \in \cD_{i,j}\) is the
relative position of adjacent generators: any monomial in the \(\X_i,\ldots,\X_j\)
such that \(\X_{l+1}\) is right of \(\X_l\) if and only if
\(l \rightarrow l+1 \in D\) is commutation equivalent to \(\X_D\).
\end{remark}

\begin{proof}[Proof of Lemma~\ref{L:splitting}]
  This follows immediately from Lemma~\ref{L:CombDesc.Xij} and
  Remark~\ref{R:good-reps} after observing that an orientation on \(P_{i,j}\)
  consists of orientations on \(P_{i,s}\) and \(P_{s+1,j}\) and an orientation
  on the edge \((s,s+1)\).
  The term \(\qr\X_{i,s}\X_{s+1,j}\) accounts for all orientations with \(s
    \rightarrow s+1\) where as \(-\qbr\X_{s+1,j}\X_{i,s}\) accounts for
  orientations with \(s \leftarrow s+1\).
\end{proof}

\subsection{Straightening in \texorpdfstring{$\Uhsl{n+1}$}{Uhsln+1}}
When straightening generators with the same subscript, all the straightening
laws developed in \S\ref{SS:Uhsl2Straightening} apply. For straightening terms
with different subscripts, we have the following set of ``mixed'' straightening
laws. The proofs of these are straightforward and are omitted.

\begin{lemma}\label{L:calcUhsl3}
  Let \(a \in \mathbb{Z}_{\geq 0}\), \(b \in \mathbb{Z}\),
  \(i,j,l \in \{1,\ldots,n\}\), and \(f(x)\) a formal power series in \(x\).
  Then the following identities hold in \(\Uhsl{n+1}\):
  \begin{align}
    \tag{i}\label{L:calcUhsl3.i}
      \X_if(\H_j) &
      =
      \begin{dcases*}
        f(\H_j + 1)\X_i, \\
        f(\H_j)\X_i,
      \end{dcases*}
        &
      \Y_if(\H_j) &
      =
      \begin{dcases*}
        f(\H_j - 1)\Y_i & if \(\abs{i - j} = 1\), \\
        f(\H_j)\Y_i & if \(\abs{i - j} > 1\),
      \end{dcases*} \\
    \tag{ii}\label{L:calcUhsl3.iii}
      \X_i^a\K_j^b &
      =
        \begin{dcases*}
          q^{\frac{1}{2}ab}\K_j^b\X_i^a, \\
          \K_j^b\X_i^a,
        \end{dcases*}
          &
      \Y_i^a\K_j^b &
      =
        \begin{dcases*}
          \qb^{\frac{1}{2}ab}\K_j^b\Y_i^a & if \(\abs{i - j} = 1\), \\
        \end{dcases*} \\
    \tag{iii}\label{L:calcUhsl3.ii}
      \X_{i,j}f(\H_l) &
      =
        \begin{dcases*}
          f(\H_l - 1)\X_{i,j}, \\
          f(\H_l - 1)\X_{i,j}, \\
          f(\H_l + 1)\X_{i,j}, \\
          f(\H_l)\X_{i,j},
        \end{dcases*}
          &
      \Y_{i,j}f(\H_j) &
      =
        \begin{dcases*}
          f(\H_l + 1)\Y_{i,j} & if \(l = i,j\), \\
          f(\H_l + 1)\Y_{i,j} & if \(j \neq i + 2\) and \(l = i+1,j-1\), \\
          f(\H_l - 1)\Y_{i,j} & if \(l = i-1,j+1\), \\
          f(\H_l)\Y_{i,j}     & otherwise,
        \end{dcases*} \\
    \tag{iv}\label{L:calcUhsl3.iv}
      \X_{i,j}^a\K_l^b &
      =
        \begin{dcases*}
          \qb^{\frac{1}{2}ab}\K_l^b\X_{i,j}^a, \\
          \qb^{\frac{1}{2}ab}\K_l^b\X_{i,j}^a, \\
          q^{\frac{1}{2}ab}\K_l^b\X_{i,j}^a, \\
          \K_l^b\X_{i,j}^a
        \end{dcases*}
          &
      \Y_{i,j}^a\K_j^b &
      =
        \begin{dcases*}
          q^{\frac{1}{2}ab}\K_l^b\Y_{i,j}^a & if \(l = i,j\), \\
          q^{\frac{1}{2}ab}\K_l^b\Y_{i,j}^a & if \(j \neq i + 2\) and \(l = i+1,j-1\), \\
          \qb^{\frac{1}{2}ab}\K_l^b\Y_{i,j}^a & if \(l = i-1,j+1\), \\
          \K_l^b\Y_{i,j}^a & otherwise,
        \end{dcases*} \\
    \tag{v}\label{L:calcUhsl3.v}
    \X_{i,j}^a\X_i^b &
    = \qb^{ab}\X_i^b\X_{i,j}^a, &
    \X_j^b\X_{i,j}^a &
    = \qb^{ab}\X_{i,j}^a\X_j^b.
  \end{align}
\end{lemma}

Interesting combinatorial structure appears when commuting \(\X_i\) through
terms like \(\X_{i+1,j}\) or \(\Y_{i,j}\).
In doing so, the relations of \(\Uhsl{n+1}\) either \emph{lengthen} or
\emph{shorten} the interval indexing \(\X_{i+1,j}\) or \(\Y_{i,j}\).

\begin{lemma}[Lengthening]\label{L:lengthenUhsln}
  Let \(1 \leq i < j \leq n\).
  Then, in \(\Uhsl{n+1}\),
  \begin{align}
    \tag{i}\label{L:lengthenUhsln.i}
      \X_{i+1,j}^{n_{i+1,j}}\X_i
        & = q^{n_{i+1,j}}\X_i\X_{i+1,j} - \qr[n_{i+1,j}]\,\X_{i,j}\X_{i+1,j}^{n_{i+1,j}-1}, \\
    \tag{ii}\label{L:lengthenUhsln.ii}
      \X_j\X_{i,j-1}^{n_{i,j-1}}
        & = q^{n_{i,j-1}}\X_{i,j-1}^{n_{i,j-1}}\X_j - \qr[n_{i,j-1}]\,\X_{i,j-1}^{n_{i,j-1} - 1}\X_{i,j}.
  \end{align}
\end{lemma}

\begin{lemma}[Shortening]\label{L:shortenUhsln}
  Let \(1 \leq i < j \leq n\). Then, in \(\Uhsl{n+1}\),
  \begin{align}
    \tag{i}\label{L:shortenUhsln.i}
      \Y_{i,j}^{n_{i,j}}\X_i
        & = \X_i\Y_{i,j}^{n_{i,j}} - \qbr [n_{i,j}]\,\K_i^2\Y_{i,j}^{n_{i,j} - 1}\Y_{i+1,j}, \\
    \tag{ii}\label{L:shortenUhsln.ii}
      \X_j\Y_{i,j}^{n_{i,j}}
        & = \Y_{i,j}^{n_{i,j}}\X_j - \qr\qb^{n_{i,j} - 1}[n_{i,j}]\,\Kb_j^2\Y_{i,j-1}\Y_{i,j}^{n_{i,j}-1}.
  \end{align}
\end{lemma}

When the exponent is \(1\), these are proven using the Splitting
Lemma~\ref{L:splitting}.
The general case is handled via the inductive technique from
Lemma~\ref{L:calcUhSL2}.

We also have the following commutation relation when the indexing interval of a
some generator is completely contained in that of another.
In these cases, all relations end up being commutation relations, possibly with
some scalar factor.
The proofs use the Splitting Lemma~\ref{L:splitting} to reduce the statement to
a computation in \(\Uhsl{4}\).

\begin{lemma}[Passing]\label{L:passingUhsln}
  Let \(1 \leq i < j \leq n\).
  Then the following hold in \(\Uhsl{n+1}\).
  \begin{align*}
    \tag{i}\label{L:passingUhsln.i} \X_{i,j}^{n_{i,j}}\X_i^{n_i} & = \qb^{n_in_{i,j}}\X_i^{n_i}\X_{i,j}^{n_{i,j}}, & \X_j^{n_j}\X_{i,j}^{n_{i,j}} & = \qb^{n_{i,j}n_j}\X_{i,j}^{n_{i,j}}\X_j^{n_j}.
  \intertext{Moreover, for any \(i < s < j\),}
    \tag{ii}\label{L:passingUhsln.ii}\X_s^{n_s}\X_{i,j}^{n_{i,j}} & = \X_{i,j}^{n_{i,j}}\X_s^{n_s}, & \X_s^{n_s}\Y_{i,j}^{n_{i,j}} & = \Y_{i,j}^{n_{i,j}}\X_s^{n_s}.
  \end{align*}
\end{lemma}

Finally, we have the following analogue of Lemma~\ref{L:1otX} for moving
generators past the exponential factor in the \(\sR\)-matrix.

\begin{lemma}\label{L:expUhsl3}
Let \(f(x)\) be a formal power series in \(x\), \(a \in \mathbf{Z}_{> 0}\),
\(\kappa \in \mathbf{Z}\) any integer and \(i,j,l \in \{1,\ldots,n\}\). Then the
following hold in \(\Uhsl{n+1}\):
\begin{align*}
      f(1 \otimes \X_i^a)e^{\kappa\frac{h}{4}\H_l \otimes \H_j}
        & = \begin{dcases*}
          e^{\kappa\frac{h}{4}\H_l \otimes \H_j} f(\Kb_l^{2a\kappa} \otimes \X_i^a), \\
          e^{\kappa\frac{h}{4}\H_l \otimes \H_j} f(\K_l^{2a\kappa} \otimes \X_i^a), \\
          e^{\kappa\frac{h}{4}\H_l \otimes \H_j} f(1 \otimes \X_i^a),
        \end{dcases*} &
      f(1 \otimes \Y_i^a)e^{\kappa\frac{h}{4}\H_l \otimes \H_j} &
        = \begin{dcases*}
          e^{\kappa\frac{h}{4}\H_l \otimes \H_j} f(\K_l^{2a\kappa} \otimes \Y_i^a), & if \(j = i\), \\
          e^{\kappa\frac{h}{4}\H_l \otimes \H_j} f(\Kb_l^{2a\kappa} \otimes \Y_i^a), & if \(\abs{j - i} = 1\), \\
          e^{\kappa\frac{h}{4}\H_l \otimes \H_j} f(1 \otimes \Y_i^a), & otherwise,
        \end{dcases*}
\end{align*}
and similarly for \(f(\X_i^n \otimes 1)\) and \(f(\Y_i^n \otimes 1)\).
\end{lemma}

\section{An \texorpdfstring{$\sR$}{R}-matrix for \texorpdfstring{$\Uhsl{n+1}$}{Uhsln+1}}\label{S:Uhsl3-coeffs}
We apply the formalism developed in \S\ref{S:Uhsln} to construct an
\(\sR\)-matrix for \(\Uhsl{n+1}\), following the ideas of \S\ref{S:Rmorphism}.
For other calculations of \(\sR\)-matrix, see~\cite{bu,kt91}.

\subsection{An ansatz}\label{RmatrixSLN1}
Reasoning similar to that in~\S\ref{SSS:R-ansatz} can be used to construct
an ansatz for \(\sR\) in the \(\Uhsl{n+1}\) case.
Alternatively, and more efficiently, we use the general principle that objects
in semisimple Lie algebras can be built by appropriately combining ingredients
from constituent \(\fsl_2\) to obtain an ansatz for \(\sR\) here.
This principle, together with the form of \(\sR\) given in
Theorem~\ref{T:Rmorphism1}, leads us to propose the following ansatz for the
form of \(\sR\) in \(\Uhsl{n+1}\):
\begin{equation}\label{e:Uhsl3.Ansatz.1}
  \sR \coloneqq \bfe^K\sum_{\bm} \alpha(\bm)\, \KK(\bm)\XX(\bm)\otimes \KKb(\bm)\YY(\bm)
\end{equation}
where: \(\bm \coloneqq (m_1,m_{1,2},m_{1,2,3},\ldots,m_{n-1,n},m_n)\) is a vector
of integers ordered as the generators are ordered in~\eqref{e:prec-order};
the coefficient \(\alpha(\bm)\) is a rational function in \(q\); the
\(\KK\), \(\KKb\), \(\XX\), and \(\YY\) are defined by
\begin{align*}
  \XX(\bm) & \coloneqq \X_1^{m_1}\X_{12}^{m_{12}} \cdots \X_{n-1,n}^{m_{n-1,n}}\X_n^{m_n},
    &
  \KK(\bm) & \coloneqq \K_1^{m_1}\K_{12}^{m_{12}} \cdots \K_n^{m_n},
    \\
  \YY(\bm) & \coloneqq \Y_1^{m_1}\Y_{12}^{m_{12}} \cdots \Y_{n-1,n}^{m_{n-1,n}}\X_n^{m_n},
    &
  \KKb(\bm) & \coloneqq \Kb_1^{m_1} \Kb_{12}^{m_{12}} \cdots \Kb_n^{m_n}.
\end{align*}
for ordered products of the generators in \(\Uhsl{n+1}\);
\(\K_{i,j} \coloneqq \K_i\K_{i+1} \cdots \K_j\); and
\[ \mathbf{e}^K \coloneqq \exp\left(\sum_{i,j = 1}^n \kappa_{i,j} \H_i \otimes \H_j\right) \]
for some matrix of coefficients \(K \coloneqq (\kappa_{ij})\).
We may write
\(\XX(m_{ij} - 1;\bm) \coloneqq \XX(m_1,m_{12},\ldots,m_{ij}-1,\ldots,m_n)\),
for example, to indicate a change in exponent of some monomial.

As in \S\ref{SS:ConstRmorph}, the parameters in~\eqref{e:Uhsl3.Ansatz.1} are
determined through Conditions~\ref{D:QTHA.i},~\ref{D:QTHA.ii},
and~\ref{D:QTHA.iii} of Definition~\ref{D:QTHA}.
As before, Condition~\ref{D:QTHA.i} completely specifies the free parameters
in~\eqref{e:Uhsl3.Ansatz.1} and the remaining conditions need to be verified
with the resulting expression.
Here, we only perform the first step and leave the latter two conditions to the
interested reader.

\begin{remark}\label{rem:Uhsl3-other-conditions}
  Verification of Conditions~\ref{D:QTHA.ii} and~\ref{D:QTHA.iii} can be done
  in a manner similar  to the \(\Uhsl{2}\) case.
  That being said, the \(\Uhsl{n+1}\) case is much more complicated due to
  non-commuting \(q\)-exponentials in~\eqref{eq:Uhsln-R}.
  Ultimately, this problem is solved by studying in detail how
  \(q\)-exponentials commute in the presence of Lengthening and Shortening.
  This is done, for example, in~\cite{kt91} through their approach.
\end{remark}

\subsection{Deriving the Coefficients}\label{SS:Uhsl3-coeffs}
It suffices to impose Condition~\ref{D:QTHA.i} for \(\X_i\), \(\H_i\) and
\(\Y_i\), \(i = 1,\ldots,n\). In fact, since all relations involving the
\(\Y_i\) are mirror to those involving the \(\X_i\), the calculations required
for the \(\Y_i\) are completely analogous to those of the \(\X_i\). So we only
need to perform computations for the \(\H_i\) and \(\X_i\).

\subsubsection{For the generators \(\H_i\)}
Since the \(\H_i\) commute with \(\K_1,\ldots,\K_n\), we need only consider
the \(\X_j\) and \(\Y_j\) components of each general term
in~\eqref{e:Uhsl3.Ansatz.1}. Write a general term as
\(\X_1^{s_1}\X_{12}^{s_{12}} \cdots \X_n^{s_n} \otimes \Y_1^{t_1}\Y_{12}^{t_{12}} \cdots \Y_n^{t_n}\).
Using Lemma~\ref{L:calcUhsl3}\eqref{L:calcUhsl3.i} and~\eqref{L:calcUhsl3.ii},
Condition~\ref{D:QTHA.i} for the \(\H_i\) are equivalent to the system of
equations
\dmrjdel{	% START DELETE
\begin{align*}
  2(s_1 - t_1) + \sum_{j = 2}^n (s_{1j} - t_{1j}) - \sum_{j = 2}^n (s_{2j} - t_{2j}) & = 0, \\
  & \;\;\vdots \\
  -\sum_{i = 1}^{n-2}(s_{i,n-1} - t_{i,n-1}) + \sum_{i = 1}^{n-1}(s_{in} - t_{in}) + 2(s_n - t_n) & = 0.
\end{align*}
} % END DELETE
%-----------------------------------------------------
$$
\left.
\begin{array}{rcl}
 2(s_1 - t_1) + \sum_{j = 2}^n (s_{1j} - t_{1j}) - \sum_{j = 2}^n (s_{2j} - t_{2j})   &= &0, \\
 & \vdots& \\
  -\sum_{i = 1}^{n-2}(s_{i,n-1} - t_{i,n-1}) + \sum_{i = 1}^{n-1}(s_{in} - t_{in}) + 2(s_n - t_n)   &=& 0.
\end{array}
\right\}
$$
\mbox{}\\[5pt]
%----------------------------------------------------
These equations are satisfied when \(s_{ij} = t_{ij} = m_{ij}\) as
in~\eqref{e:Uhsl3.Ansatz.1}.

\subsubsection{For the generators \(\X_i\)}
Since \(\Delta(\X_i) = \X_i \otimes \K_i + \Kb_i \otimes \X_i\),
the Condition reads
\begin{equation}\label{e:Condition.1}
  (\K_i\otimes\X_i + \X_i\otimes\Kb_i) \cdot \sR
    = \sR \cdot (\X_i\otimes\K_i + \Kb_i\otimes\X_i).
\end{equation}
Following Remark~\ref{R:Observe}\hyperref[R:Observe.B]{(B)},
decompose the general term of~\eqref{e:Condition.1} by setting
\begin{equation}\label{e:AiBi'}
  \begin{aligned}
    A_i^+(\bm) & \coloneqq (\K_i \otimes \X_i) \cdot \bfe^K\, \KK(\bm)\XX(\bm) \otimes \KKb(\bm)\YY(\bm), &
    B_i^-(\bm) & \coloneqq (\X_i \otimes \Kb_i) \cdot \bfe^K\,\KK(\bm)\XX(\bm) \otimes \KKb(\bm)\YY(\bm), \\
    A_i^-(\bm) & \coloneqq \bfe^K\, \KK(\bm)\XX(\bm)\otimes \KKb(\bm)\YY(\bm) (\Kb_i \otimes \X_i), &
    B_i^+(\bm) & \coloneqq \bfe^K\,\KK(\bm)\XX(\bm)\otimes \KKb(\bm)\YY(\bm) (\X_i \otimes \K_i),
  \end{aligned}
\end{equation}
so~\eqref{e:Condition.1} is the assertion
\[ \sum_\bm A_i^+(\bm) + B_i^-(\bm) = \sum_\bm A_i^-(\bm) + B_i^+(\bm). \]
As in~\eqref{eq:AeqB}, we rearrange the sums so that equation~\eqref{e:Condition.1}
is equivalent to
\[ A_i \coloneqq \sum_\bm A_i^+(\bm) - A_i^-(\bm) = \sum_\bm B_i^+(\bm) - B_i^-(\bm) \eqqcolon B_i. \]
We wish to simplify the series on both sides.
In the \(\Uhsl{2}\) case, the essential simplification to the \(A\)-side came
in identifying a commutator \([\X,\Y^m]\) in~\eqref{e:Aform2}.
Analogously, if we straighten and apply the Shortening
Lemma~\ref{L:shortenUhsln} to the terms in \(A_i^+(\bm) - A_i^-(\bm)\), we find
a single term with \(\cdots \X_i\Y_i^{m_i}\cdots\) in \(A_i^+\) and a single
term with \(\cdots \Y_i^{m_i}\X_i \cdots\) in \(A_i^-\) which should be
combined to form a commutator \([\X_i,\Y_i^{m_i}]\).

\subsubsection{Exponential Prefactor}
To start combining \(A_i^+(\bm)\) with \(A_i^-(\bm)\), the exponential factor
\(\bfe^K\) of both terms should coincide after straightening.
From~\eqref{e:AiBi'}, this means that we should have
\((\K_i \otimes \X_i) \bfe^K = \bfe^K (\Kb_i \otimes \X_i)\).
Using Lemma~\ref{L:calcUhsl3}\eqref{L:calcUhsl3.iii} and Lemma~\ref{L:expUhsl3},
we obtain the following system of \(n\) linear equations:
\begin{align*}
  \kappa_{l,i-1} - 2\kappa_{l,i} + \kappa_{l,i+1} = -2\delta_{li}, \qquad l = 1,\ldots,n,
\end{align*}
where \(\kappa_{l,-1} \coloneqq 0\) and \(\kappa_{l,n+1} \coloneqq 0\).
This means \(K = (\kappa_{i,j})\) is twice the inverse of the Cartan matrix of
\(\fsl_{n+1}\), which can be computed to be~\cite[p.69]{hum72}
\begin{equation}\label{e:kappa}
  \kappa_{i,j} = \begin{cases}
    \frac{2}{n + 1}i(n - j + 1) & \text{if \(j \leq i\),} \\
    \frac{2}{n + 1}j(n - i + 1) & \text{if \(i \leq j\).}
  \end{cases}
\end{equation}

In the following, we may pull the exponential factor out of \(A_i\) and
\(B_i\), at which point it remains to straighten the terms of the form
\((\K_i \otimes \X_i) \KK(\bm)\XX(\bm) \otimes \KKb(\bm)\YY(\bm)\),
in order to deduce a recurrence relation on the coefficients \(\alpha(\bm)\).
Because of Shortening and Lengthening phenomena, each summand of \(A_i\) and
\(B_i\) will itself be a sum of terms indexed by certain segments of
\(\{1,\ldots,n\}\).
To be clear, \emph{straightening} the term \(\X_i\YY\) means that we need to
move \(\X_i\) through the \(\YY\) until it is immediately left of \(\Y_i\);
similarly, \emph{straightening} \(\YY\X_i\) means that \(\X_i\) needs to be
moved until it is immediately right of \(\Y_i\).

\subsection{The $A_i^\pm$ terms}

\subsubsection{Terms of \(A_i^+ = (\K_i \otimes \X_i) \cdot \sR\)}
Straightening \(A_i^+\) first involves moving \(\X_i\) past the \(\KKb\) term
in front of \(\YY\) in the second factor, and
Lemma~\ref{L:calcUhsl3}\eqref{L:calcUhsl3.iii} tells us that this
contributes a \(q\) coefficient for each \(\Kb_{a,b}\) such that
either one or two elements of \(\{i-1,i,i+1\}\) is between \(a\) and \(b\);
precisely, the contribution when \(\X_i\) is moved past
\begin{itemize}
  \item \(\Kb_{a,i-1}\) is \(\qb^{\frac{1}{2}m_{a,i-1}}\) and
    \(\Kb_{a,i}\) is \(q^{\frac{1}{2}m_{a,i}}\) for \(1 \leq a < i\);
  \item \(\Kb_i\) is \(q^{m_i}\); and
  \item \(\Kb_{i,b}\) is \(q^{\frac{1}{2}m_{i,b}}\) and
    \(\Kb_{i+1,b}\) is \(\qb^{\frac{1}{2}m_{i+1,b}}\) for \(i < b \leq n\).
\end{itemize}
Thus the terms of \(A_i^+\) obtain an overall coefficient contribution of
\begin{equation}\label{e:overall.qi}
  q^{m_i}
  \prod_{a = 1}^{i - 1} q^{\half m_{a,i} - \half m_{a,i-1}}
  \prod_{b = i+1}^n q^{\half m_{i,b} - \half m_{i+1,b}}.
\end{equation}

Next, \(\X_i\) must be commuted, from the left, through terms \(\Y_{s,i}\) for
\(s = 1,\ldots,i-1\), and each commutation creates a new term \(A_{s,i}^+\)
in which \(\Y_{s,i}^{m_{s,i}}\) is shortened.
The Shortening Lemma~\ref{L:shortenUhsln}\eqref{L:shortenUhsln.ii} shows that
the \(\YY\) term in \(A_{s,i}^+\) has the form
\[
  \big(\cdots \Y_{s,i-1}^{m_{s,i-1}}\Y_{s,i}^{m_{s,i}} \cdots\big)
    \mapsto
  \big(\cdots \Y_{s,i-1}^{m_{s,i-1}} \big(-\qr\qb^{m_{s,i} - 1}[m_{s,i}]\,\Kb_i^2\Y_{s,i-1}\Y_{s,i}^{m_{s,i} - 1}\big)\cdots\big).
\]
Coefficients of \(q\) are accrued in \(A_{s,i}^+\) by straightening the
\(\Kb_i^2\) past
\begin{itemize}
  \item \(\Y_{a,i-1}\) with contribution \(q^{m_{a,i-1}}\) for \(1 \leq a \leq s\); and
  \item \(\Y_{a,i}\) with contribution \(\qb^{m_{a,i}}\) for \(1 \leq a < s\).
\end{itemize}
In total, the coefficient of \(A_{s,i}^+\), \(1 \leq s < i\), is
\begin{equation}\label{e:Coef.Asi}
  -\alpha[m_{s,i}]\q^{\frac{3}{2} + m_i}
    \prod_{a = 1}^s \qb^{\half m_{a,i} - \half m_{a,i-1}}
    \prod^{i-1}_{a = s + 1} \q^{\half m_{a,i} - \half m_{a,i-1}}
    \prod_{b = i + 1}^n \q^{\half m_{i,b} - \half m_{i+1,b}}.
\end{equation}
Moreover, the exponents of the generators in \(A_{s,i}^+\) change as
\begin{equation}\label{e:Powers.Asi}
  \KK(\bm)\XX(\bm)\otimes\KKb(\bm)\YY(\bm)
    \mapsto \KK(\bm)\XX(\bm)\otimes\KKb(\bm)\YY(m_{s,i-1} + 1, m_{s,i}-1;\bm).
\end{equation}

Finally, let \(A_{i,i}^+\) be the remaining term in which \(\X_i\) has been
moved to the immediate left of \(\Y_i\), i.e. the term obtained by repeatedly
taking the first term in Lemma~\ref{L:shortenUhsln}\eqref{L:shortenUhsln.ii}.
The coefficient of \(A_{i,i}^+\) is simply~\eqref{e:overall.qi}.

\subsubsection{Terms of \(A_i^- = \sR\cdot (\Kb_i \otimes \X_i)\)}
Straightening the \(\Kb_i\) through the first tensor factor gives an overall
\(q\) coefficient contribution, which can be computed to
be~\eqref{e:overall.qi} again.
Straightening the \(\X_i\) term on the second factor involves moving past
terms \(\Y_{i,t}\) for \(i \leq t < n\), and the Shortening
Lemma~\ref{L:shortenUhsln}\eqref{L:shortenUhsln.i} shows that this creates a
new term \(A_{i,t}^-\) whose \(\YY\) factor has the form
\[
  \big(\cdots\Y_{i,t}^{m_{i,t}}\Y_{i,t+1}^{m_{i,t+1}}\cdots\big)
  \mapsto
  \big(\cdots\big(-\qbr[m_{i,t}]\,\K_i^2\Y_{i,t}^{m_{i,t} - 1}\Y_{i+1,t}\big)\Y_{i,t+1}^{m_{i,t+1}}\cdots\big).
\]
In straightening \(A_{i,t}^-\), the \(\K_i^2\) needs to be moved left and
the \(\Y_{i+1,t}\) needs to be moved right, so using the
Passing Lemma~\ref{L:passingUhsln}, the coefficient of \(A_{i,t}^-\) is
\begin{equation}\label{e:Coef.Ait}
  -\alpha(\bm)[m_{i,t}]q^{3m_i - \half}
    \prod_{a = 1}^{i - 1} \q^{\frac{3}{2}m_{a,i} - \frac{3}{2}m_{a,i-1}}
    \prod_{b = i + 1}^{t - 1}\q^{\frac{3}{2}m_{i,b} - \frac{3}{2}m_{i+1,b}}
    \prod_{b = t}^n \q^{\half m_{i,b} - \half m_{i+1,b}}
\end{equation}
and that the exponents in \(A_{i,t}^-\) change as
\begin{equation}\label{e:Powers.Ait}
  \KK(\bm)\XX(\bm)\otimes\KKb(\bm)\YY(\bm)
    \mapsto \KK(\bm)\XX(\bm)\otimes\KKb(\bm)\YY(m_{i,t} - 1, m_{i+1,t} + 1;\bm).
\end{equation}

As before, writing \(A_{i,i}^-\) for the remaining term where \(\X_i\) is
immediately right of \(\Y_i\), the coefficient here is just~\eqref{e:overall.qi}.

\subsection{The $B_i^\pm$ terms}
\subsubsection{Terms of \(B_i^+ = (\sR \cdot (\X_i \otimes \K_i))\)}
Straightening \(\K_i\) through \(\YY\) gives a coefficient
of~\eqref{e:overall.qi} for \(B_i^+\).
In straightening \(\X_i\) in the first factor, the Lengthening
Lemma~\ref{L:lengthenUhsln}\eqref{L:lengthenUhsln.i} shows that commuting past
\(\X_{i+1,t}\), \(1 \leq t < n\), creates a new term \(B_{i,t}^+\) of the form
\[
  \big(\cdots\X_{i+1,t-1}^{m_{i+1,t-1}}\X_{i+1,t}^{m_{i+1,t}}\cdots\big)
    \mapsto
  \big(\cdots \X_{i+1,t-1}^{m_{i+1,t-1}}\big(-\qr[m_{i+1,j}]\,\X_{i,t}\X_{i+1,t}^{m_{i+1,t} - 1}\big)\cdots\big).
\]
Straightening this involves moving \(\X_{i,t}\) to the left through
\(\X_{i+1,t'}\) for \(t' < t\), and then through \(\X_{i,t''}\) for \(t < t''\).
By the Passing Lemma~\ref{L:passingUhsln}\eqref{L:passingUhsln.ii},
there are no additional factors of \(q\) when moving past the \(\X_{i+1,b}\)
terms; by Lemma~\ref{L:passingUhsln}\eqref{L:passingUhsln.i} again, a factor of
\(\qb^{m_{i,b}}\) is obtained when moving past \(\X_{i,b}^{m_{i,b}}\).
The coefficient of \(B_{i,t}^+\) is thus
\begin{equation}\label{e:Coef.Bit}
  -\alpha(\bm)[m_{i,t}]\q^{m_i + \half}
    \prod_{a = 1}^{i - 1}\q^{\half m_{a,i} - \half m_{a,i-1}}
    \prod_{b = i+1}^t \q^{\half m_{i,b} - \half m_{i+1,b}}
    \prod_{b = t+1}^n \qb^{\half m_{i,b} + \half m_{i+1,b}}
\end{equation}
and the exponents in \(B_{i,t}^+\) change as
\begin{equation}\label{e:Powers.Bit}
  \KK(\bm)\XX(\bm)\otimes\KKb(\bm)\YY(\bm)
    \mapsto \KK(\bm)\XX(m_{i,t} + 1, m_{i+1,t} - 1; \bm)\otimes\KKb(\bm)\YY(\bm).
\end{equation}

This time, the remaining term \(B_{i,i}^+\) picks up a coefficient from the
Lengthening Lemma~\ref{L:lengthenUhsln} and from the
Passing Lemma~\ref{L:passingUhsln}\eqref{L:passingUhsln.i}, giving an overall
coefficient
\begin{equation*}
  \alpha(\bm)\q^{m_i}
    \prod_{a = 1}^{i - 1}\q^{\half m_{a,i} - \half m_{a,i-1}}
    \prod_{b = i+1}^n \qb^{\half m_{i,b} - \half m_{i+1,b}}.
\end{equation*}

\subsubsection{Terms of \(B_i^- = (\X_i \otimes \Kb_i) \cdot \sR\)}
Straightening \(\X_i\) through \(\KK\) in the first tensor factor gives a
coefficient contribution inverse of~\eqref{e:overall.qi} to \(B_i^-\).
It remains to move \(\X_i\) through the \(\XX\) from the left, in which case
lengthening occurs for every term \(\X_{s,i-1}\), \(1 \leq s < i\), yielding a
new term \(B_{s,i}^-\), which by the Lengthening Lemma~\ref{L:lengthenUhsln}\eqref{L:lengthenUhsln.ii}
has the form
\[
  \big(\cdots \X_{s,i-1}^{m_{s,i-1}}\X_{s,i}^{m_{s,i}} \cdots\big)
  \mapsto
  \big(\cdots \big(-\qr[m_{s,i-1}]\,\X_{s,i-1}^{m_{s,i-1} - 1}\X_{s,i}\big)\X_{s,i}^{m_{s,i}}\cdots\big).
\]
No further straightening is required in \(B_{s,i}^-\).
Note, however, that in order to get \(\X_i\) to \(\X_{s,i-1}\) in the first place,
it needs to pass through terms \(\X_{a,i-1}^{m_{a,i-1}}\) and
\(\X_{a,i}^{m_{a,i}}\) for \(1 \leq a < s\), which by the Passing
Lemma~\ref{L:passingUhsln}\eqref{L:passingUhsln.i}, give coefficient
contributions \(\q^{m_{a,i-1}}\) and \(\qb^{m_{a,i}}\), respectively.
Thus the coefficient of \(B_{s,i}^-\) is
\begin{equation}\label{e:Coef.Bsi}
  -\alpha(\bm)[n_{s,i-1}]\qb^{m_i - \half}
    \prod_{i = 1}^{s - 1}\qb^{\frac{3}{2}m_{a,i} - \frac{3}{2}m_{a,i-1}}
    \prod_{a = s}^{i - 1}\qb^{\half m_{a,i} - \half m_{a,i-1}}
    \prod_{b = i+1}^n \qb^{\half m_{i,b} - \half m_{i+1,b}}
\end{equation}
and the exponents in \(B_{s,i}^-\) change as
\begin{equation}\label{e:Powers.Bsi}
  \KK(\bm)\XX(\bm)\otimes\KKb(\bm)\YY(\bm)
    \mapsto \KK(\bm)\XX(m_{s,i-1} - 1, m_{s,i} + 1;\bm) \otimes \KKb(\bm)\YY(\bm).
\end{equation}

Let \(B_{i,i}^-\) the remaining term in which the exponent of \(\X_i\) is
increased by \(1\).
Then by reasoning as in the case of \(B_{i,i}^+\), we see the coefficient of
this term is
\begin{equation*}
  \alpha(\bm)\qb^{m_i}
    \prod_{a = 1}^{i - 1}\qb^{\frac{3}{2} m_{a,i} - \frac{3}{2} m_{a,i-1}}
    \prod_{b = i+1}^n \qb^{\half m_{i,b} - \half m_{i+1,b}}.
\end{equation*}

\subsection{Combining Diagonal Terms}
We now combine the diagonal terms of \(A_i\) and \(B_i\).
These simplifications should be compared to those made when computing \(\sR\)
for \(\Uhsl{2}\) in \S\ref{S:Rmorphism}.
First, the diagonal \(A\) terms \(A_{i,i}^+(\bm)\) and \(A_{i,i}^-(\bm)\)
differ only in the \(\YY\) component as
\begin{align*}
  A_{i,i}^+ & = \KK\XX \otimes \KKb(\cdots \X_i\Y_i^{m_i} \cdots), &
  A_{i,i}^- & = \KK\XX \otimes \KKb(\cdots \Y_i^{m_i}\X_i \cdots).
\end{align*}
These terms may be combined in \(A_i\) to obtain a commutator
\([\X_i,\Y_i^{m_i}]\), which can be simplified using
Lemma~\ref{L:calcUhSL2}\eqref{L:calcUhSL2.v}; denote the resulting term by
\(A_{i,i}\).
This term has coefficient
\begin{equation}\label{e:Coef.Aii}
  \alpha(\bm)[m_i]\,\q^{m_i}
    \prod_{a = 1}^{i - 1} \q^{\half m_{a,i} - \half m_{a,i-1}}
    \prod_{b = i+1}^n \q^{\half m_{i,b} - \half m_{i+1,b}}
\end{equation}
and exponent change of
\begin{equation}\label{e:Powers.Aii}
  \KK(\bm)\XX(\bm)\otimes\KKb(\bm)\YY(\bm)
    \mapsto \KK(\bm)\XX(\bm) \otimes \KKb(\bm)\YY(m_i - 1;\bm).
\end{equation}
Similarly, \(B_{i,i}^+(\bm)\) and \(B_{i,i}^-(\bm)\) can be combined by
factoring out
\[
  (q - \qb)\,
    \prod_{a = 1}^{i - 1}\qb^{\half m_{a,i} - \half m_{a,i-1}}
    \prod^n_{b = i+1} \qb^{\half m_{i,b} - \half m_{i+1,b}}
\]
from each of \(B_{i,i}^\pm(\bm)\).
This allows us to combine the generators in \(B_{i,i}^\pm(\bm)\), with
\[
  \frac{1}{q - \qb}
    \Big(\Big(\prod^{i - 1}_{a = 1}q^{m_{a,i} - m_{a,i-1}}\Big)q^{m_i} \Kb_i^{m_i - 1} - \Big(\prod^{i - 1}_{a = 1}\qb^{m_{a,i} - m_{a,i-1}}\Big)\qb^{m_i} \Kb_i^{m_i + 1}\Big)
\]
in place of the \(\Kb_i\) in \(\KKb\).
After factoring out \(\Kb_i^{m_i+1}\), this can be simplified to
\[ \big[\H_i + m_i + \sum_{a = 1}^{i - 1} m_{a,i} - m_{a,i-1}\big]. \]
After straightening, this will contribute the same quantum number as obtained
from simplifying the commutator in \(A_{i,i}\).
Overall, the coefficient of \(B_{i,i}\) is
\begin{equation}\label{e:Coef.Bii}
  \alpha(\bm)(q - \qb)
    \prod_{a = 1}^{i - 1}\qb^{\half m_{a,i} - \half m_{a,i-1}}
    \prod^n_{b = i+1} \qb^{\half m_{i,b} - \half m_{i+1,b}}
\end{equation}
and the change in exponent is due only to the additional \(\X_i\) term on the
right
\begin{equation}\label{e:Powers.Bii}
  \KK(\bm)\XX(\bm)\otimes\KKb(\bm)\YY(\bm)
    \mapsto \KK(\bm)\XX(m_i + 1;\bm) \otimes \KKb(\bm)\YY(\bm).
\end{equation}

\subsection{Recurrence for Coefficients}
Finally, a recurrence for the \(\alpha(\bm)\) is constructed by comparing like
terms in \(A_i = B_i\).
From~\eqref{e:Powers.Aii} and~\eqref{e:Powers.Bii}, the terms of \(A_{i,i}\)
agree in shape with those of \(B_{i,i}\) so coefficients can be compared after
making the shift \(m_i \mapsto m_i - 1\) in \(B_{i,i}\).
Equations~\eqref{e:Coef.Aii} and~\eqref{e:Coef.Bii} then give a recursion
relation for \(\alpha(\bm)\) with respect to the index \(m_i\):
\begin{equation}\label{e:Rec.Diagonal}
  \alpha(\bm)
    = \alpha(m_i - 1;\bm)\frac{(\q - \qb)}{[m_i]}\, \qb^{m_i}
      \prod_{a = 1}^{i - 1} \qb^{m_{a,i} - m_{a,i-1}}
      \prod_{b = i + 1}^n \qb^{m_{i,b} - m_{i+1,b}}.
\end{equation}

From~\eqref{e:Powers.Asi} and~\eqref{e:Powers.Bsi}, the coefficients of
\(A_{s,i}^+\) and \(B_{s,i}^-\) can be compared upon shifting
\(m_{s,i-1} \mapsto m_{s,i-1}\) for \(A_{s,i}^+\) and \(m_{s,i} \mapsto m_{s,i} - 1\) for
\(B_{s,i}\).
Equations~\eqref{e:Coef.Asi} and~\eqref{e:Coef.Bsi} then yield a recursion of
the form
\begin{equation}\label{e:Rec.General}
  \alpha(m_{s,i}-1;\bm) =
  -\alpha(m_{s,i-1} - 1;\bm)\frac{[m_{s,i}]}{[m_{s,i-1}]}\,
    q^{2m_i - m_{s,i} + m_{s,i-1}}
    \prod_{a = 1}^{i-1} \q^{m_{a,i} - m_{a,i-1}}
    \prod_{b = i + 1}^n \q^{m_{i,b} - m_{i+1,b}}.
\end{equation}
Comparing~\eqref{e:Powers.Ait} and~\eqref{e:Powers.Bit} shows that the
coefficients of \(A_{i,t}^-\) and \(B_{i,t}^+\) can be compared, after making
an appropriate exponent shift, and similar type of recursion relation can be
constructed using~\eqref{e:Coef.Ait} and~\eqref{e:Coef.Bit}.
However, the two sets of relations~\eqref{e:Rec.Diagonal}
and~\eqref{e:Rec.General} are sufficient to solve for \(\alpha(\bm)\).

\subsection{Solving the Recurrences}\label{SS:solve-recurrence}
The recursion~\eqref{e:Rec.Diagonal} can be solved directly.
For~\eqref{e:Rec.General}, notice that the relation expresses the changes in
\(\alpha(\bm)\) with respect to \(m_{s,i}\) in terms of changes with respect to
\(m_{s,i-1}\), an exponent indexed by a shorter interval.
Thus~\eqref{e:Rec.General} relates \(\alpha(m_{s,i} - 1)\) with
\(\alpha(m_{s,i-1} - 1)\), which, in turn, is related to \(\alpha(m_{s,i-2} - 1)\),
and so forth.
Ultimately, we obtain a relation between \(\alpha(m_{s,i}-1)\) and
\(\alpha(m_{s,s} - 1) = \alpha(m_s - 1)\).
Finally,~\eqref{e:Rec.Diagonal} can be applied to relate \(\alpha(m_{s,i} - 1)\)
with \(\alpha(\bm)\).
This process yields the following recurrence:
\begin{equation}\label{e:Rec.General.II}
  \alpha(\bm) =
  (-1)^{i - s}\alpha(m_{s,i} - 1;\bm) \frac{(\q - \qb)}{[m_{s,i}]}\q^{m_{s,i}}
    \prod^i_{a = 1}\qb^{m_{a,i}}
    \prod^n_{b = i+1}\q^{m_{i+1,b}}
    \prod^{s-1}_{a = 1}\q^{m_{a,s-1}}
    \prod^n_{b = s} \qb^{m_{s,b}}.
\end{equation}
Solving~\eqref{e:Rec.Diagonal} and~\eqref{e:Rec.General.II} separately and
putting the results together give
\[
  \alpha(\bm) =
    \prod^n_{a = 1}\prod^n_{b = 1}
      (-1)^{b - a} \frac{(q - \qb)^{m_{a,b}}}{[m_{a,b}]!}\,q^{\half m_{a,b}(m_{a,b} - 3)}
      \times (\text{cross terms})
\]
where the cross terms are indexed by pairs \(s < i\) and are of the form
\[
  \Big(
    \prod^i_{\substack{a = 1 \\ a \neq s}}\qb^{m_{s,i}m_{a,i}}
    \prod^n_{b = i+1}\q^{m_{s,i}m_{i+1,b}}
  \Big)
  \Big(
    \prod^{s-1}_{a = 1}\q^{m_{a,s-1}m_{s,i}}
    \prod^n_{\substack{b = s \\ b \neq i}} \qb^{m_{s,b}m_{s,i}}
  \Big).
\]
These cross terms can be eliminated by rearranging the general term of \(\sR\)
\[
  \KK(\bm)\XX(\bm) \otimes \KKb(\bm)\YY(\bm)
    \mapsto
    \big((\K_1\X_1)^{m_1}(\K_{12}\X_{12})^{m_{12}} \cdots\big)
      \otimes
    \big((\Kb_1\Y_1)^{m_1}(\Kb_{12}\Y_{12})^{m_{12}} \cdots\big)
\]
to consist of powers of \(\K_{a,b}\X_{a,b}\) and \(\Kb_{a,b}\Y_{a,b}\).
The factors of \(q\) arising from this rearrangement cancel the cross terms,
leaving only the first products in the expression of \(\alpha(\bm)\) above.

\subsection{An \(\sR\)-matrix for \(\Uhsl{n+1}\)}\label{RmatrixSLN2}
All together, the computations of \S\S\ref{SS:Uhsl3-coeffs}--\ref{SS:solve-recurrence}
show that
\begin{multline*}
  \sR = \exp\Bigg(\frac{h}{4}\sum_{i,j = 1}^n \kappa_{ij}\H_i \otimes \H_j\Bigg)
    \sum_\bm
    \Bigg[
      \Big(\prod^n_{a = 1}\prod^n_{b = 1} (-1)^{b - a} \frac{(q - \qb)^{m_{a,b}}}{[m_{a,b}]!}\,q^{\half m_{a,b}(m_{a,b} - 3)}\Big) \\
        \times
      \Big((\K_1\X_1)^{m_1} (\K_{12}\X_{12})^{m_{12}} \cdots (\K_n\X_n)^{m_n} \Big)
        \otimes
      \Big((\Kb_1\Y_1)^{m_1})(\Kb_{12}\Y_{12})^{m_{12}} \cdots (\Kb_n\Y_n)^{m_n}\Big)
    \Bigg],
\end{multline*}
where the \(\kappa_{ij}\) are as in~\eqref{e:kappa}.
Each summand is a product of \(q\)-exponential functions, as in
\S\ref{SS:q-exp}, so
\begin{equation}\label{eq:Uhsln-R}
  \sR =
  \exp\Bigg(\frac{h}{4}\sum_{i,j = 1}^n \kappa_{ij}\H_i \otimes \H_j\Bigg)
    \prod^\prec_{1 \leq a,b \leq n}
      \Exp{q}\Big((-1)^{b - a}\lambda_q\, \K_{a,b}\X_{a,b} \otimes \Kb_{a,b}\Y_{a,b}\Big)
\end{equation}
for an \(\sR\)-matrix for \(\Uhsl{n+1}\). Here,
\(\lambda_q = \qb(q - \qb)\) as in~\eqref{e:Ras:qExpFn} and the
\(\prec\) on the product signifies that terms in the product are taken with
respect to the ordering of indices in~\eqref{e:prec-order}.

\bibliographystyle{amsalpha}
\bibliography{HolfAlgSTR}
\end{document}